\documentclass{amsart}
\usepackage{amsmath, amssymb, mathrsfs, verbatim, multirow}

\usepackage[all]{xy}
\usepackage{pifont}
\usepackage{float}
\usepackage{color}
\usepackage{enumitem}

\usepackage{tikz}
\usepackage{tikz-cd}
\usepackage{tkz-tab}
\usepackage{subcaption}

\usepackage[hidelinks]{hyperref}

\newtheorem{Teo}{Theorem}[section]
\newtheorem{Prop}[Teo]{Proposition}
\newtheorem{Lema}[Teo]{Lemma}
\newtheorem{Cor}[Teo]{Corollary}

\newtheorem{Result}[Teo]{Main Result}

\theoremstyle{theorem}
\newtheorem{Def}[Teo]{Definition}

\newtheorem{Obs}[Teo]{Remark}

\newcommand{\Z}{\mathbb{Z}}
\newcommand{\N}{\mathbb{N}}

\newcommand{\lra}{\longrightarrow}

\newcommand{\VR}{\mathcal{O}}

\newcommand{\MI}{\mathfrak{m}}

\newcommand{\supp}{\mbox{\rm supp}}

\newcommand{\K}{\mathbb{K}}

\newcommand{\Char}{\mbox{\rm char}}

\DeclareMathOperator{\inv}{in}
\begin{document}
\title[Tame fields, Graded Rings and Finite Complete Sequences]{Tame fields, Graded Rings and Finite Complete Sequences of Key Polynomials}
\author{Silva de Souza, C. H.}
%\thanks{During the realization of this project the author was supported by a grant from Funda\c c\~ao de Amparo \`a Pesquisa do Estado de S\~ao Paulo (process numbers 2021/13531-0 and 2023/04651-8).}

%
\begin{abstract} In this paper, we present a criterion for $(K,v)$ to be henselian and defectless in terms of finite complete sequences of key polynomials. For this, we use the theory of Mac Lane-Vaquié chains and abstract key polynomials.  We  then prove that a valued field $(K,v)$ is tame if and only if $vK$ is $p$-divisible, $Kv$ is perfect and every simple algebraic extension of $K$ admits a finite complete sequence of key polynomials. The properties $vK$ $p$-divisible and $Kv$ perfect are described by the Frobenius endomorphism on the associated graded ring. We also make considerations on simply defectless and algebraically maximal valued fields and purely inertial and purely ramified extensions.

\end{abstract}


\maketitle

\section{Introduction}
 One of the main obstacles in the programs for solving the \textit{Local Uniformization Problem} in positive characteristic (which is the local version of \textit{resolution of singularities}) is the \textit{defect} (see \cite{CutkoskyDefectLocalUni} and \cite{kulhmannlocalunif}). The defect of an extension has been vastly studied in the recent years (see \cite{CutkoskyDefectLocalUni},  \cite{kulhmannlocalunif},
 \cite{NartJosneiDefectFormula} and \cite{Vaqfamilleadmissible}).  \textit{Tame fields}, defined in Section \ref{TameFields},   can be seen as a better environment to do Valuations Theory and  to achieve local uniformization. For instance,  Knaf and Kuhlmann prove in \cite{KK2}  that every valuation admits local uniformization in a separable extension of the function field. The main tool used to achieve this result was the theory of tame fields. For more on tame fields, the reader may see \cite{kuhlmannTameFields}.
 % and \cite{DeeplyRamifiedKR}. 
 
 \vspace{0.3cm}
 
 There are several research programs aiming to achieve local uniformization in  positive characteristic (see \cite{KK}, \cite{JN_1} and \cite{Tei}). One of these programs has \textit{key polynomials} as one of its main tools. This concept was developed by Mac Lane in \cite{MacLane} and generalized by Vaquié in \cite{Vaq}. Recently, Novacoski and Spivakovsky in \cite{josneiKeyPolyPropriedades}  and Decaup, Mahboub and  Spivakovsky in \cite{spivamahboubkeypoly} introduced a new notion of key polynomial.  The structure of graded ring associated to a valuation is crucial in the study of these polynomials, together with the so called \textit{complete sets of key polynomials} (for definitions see Sections \ref{GradedRingsKx} and \ref{CompleteSets}).  
 
 \vspace{0.3cm}
 
 In this paper we present a link between the concept of tame field and the study of graded rings and complete sequences of key polynomials. Connections between tame fields and key polynomials were already studied for example in \cite{duttakuhlmannTameKeyPoly}, where complete sequences of key polynomials are constructed over tame fields. Here, we propose a necessary and sufficient condition for a valued  field to be tame in terms of the mentioned objects.
 
 \vspace{0.3cm}

In order to achieve this characterization of tame fields, we will study necessary and sufficient    conditions for a simple valued field  extension  to admit a finite complete sequence of key polynomials. For a simple algebraic valued field extension  $(L\mid K, v)$, with $L=K(\eta)$, consider the induced valuation $\nu(f):=v(f(\eta))$ on $K[x]$. Denote by $\Psi_m$ the set of all key polynomials of the same degree $m$.
% (this valuation does not depend on the generator of the simple extension).
 We will say that $\Psi_m$ \textit{admits a maximum} (with respect to $\nu$) if $\nu(\Psi_m):=\{\nu(Q)\mid Q\in \Psi_m\}$ admits a maximal element.
 
  \vspace{0.3cm}

The key to prove our main result will be Proposition \ref{lemSeqFinitaComplEPsimCommax}: it states that $\nu$ in the above conditions admits a finite complete sequence of key polynomials if and only if for every $m\geq 1$ either $\Psi_m$ is empty or it admits a maximum.
  In particular, this means that $\nu$ admits a complete sequence without \textit{limit key polynomials}.

\vspace{0.3cm}

We will say that $(L\mid K, v)$ is unibranched if the extension of $v$ from $K$ to $L$ is unique. Denoting by $d(L\mid K,v)$ the defect of the extension (see definition in Section \ref{ValDef}), our main result is the following.

\begin{Result}(Theorem \ref{propFiniteSeqiffDefectlessUnibranched})
	Let $(L\mid K, v)$ be a simple algebraic extension
	%	, take $\eta$ a generator for $L\mid K$ 
	and
	let $\nu$ be the induced valuation on $K[x]$.
	%defined by $\eta$ and $v$.
	Then $\nu$ admits a finite complete sequence of key polynomials if and only if $d(L\mid K, v)=1$ and $(L\mid K, v)$ is unibranched.
\end{Result}

The main reasoning used to prove Theorem \ref{propFiniteSeqiffDefectlessUnibranched} is that  a finite complete sequence of key polynomials gives us a \textit{Mac Lane-Vaquié chain of augmentations}  formed only by \textit{ordinary augmentations} (see definitions in Section \ref{Augmentations}). In one direction, we use the multiplicative property of the \textit{relative  ramification index} and the \textit{relative inertial degree} in a chain of valuations to prove that $[L:K]=(vL:vK)[Lv:Kv]$. In the other direction, we consider the  chain of valuations given by the \textit{truncations} $\nu_Q$ (see definition in Section \ref{CompleteSets}), where $Q$ is an element of a complete sequence $\boldsymbol{Q}$ for $\nu$. We use a formula for the defect to show that this chain of augmentations
% $\{(\nu_Q, Q, \nu(Q))  \}_ {Q\in \boldsymbol{Q}}$
 must be finite and hence also $\boldsymbol{Q}$ is finite.
 
  \vspace{0.3cm}
 
 As a first consequence of Theorem
 \ref{propFiniteSeqiffDefectlessUnibranched}, we have Corollary \ref{teoHensDefectlessiffKjfiniteCompleteseq}. For a  finite extension $(L\mid K, v)$, with $L=K(a_1,\ldots, a_m)$, we write $K_j = K(a_1, \ldots, a_j)$, $1\leq j\leq m$ and consider the simple extension $(K_j\mid K_{j-1},v)$ (with $K_0=K$). Corollary \ref{teoHensDefectlessiffKjfiniteCompleteseq} states that  $(K,v)$ is defectless and henselian (see definitions in Section \ref{ValDef}) if and only if for every finite extension $(L\mid K, v)$ each simple subextension $(K_j\mid K_{j-1},v)$ admits a finite complete sequence of key polynomials.
 
  \vspace{0.3cm}
  
  We then deduce
%  from Theorem \ref{propFiniteSeqiffDefectlessUnibranched} and \cite[Theorem 3.2]{kuhlmannTameFields} 
   the following characterization of tame fields:  
   $(K,v)$ is a tame field if and only if  $vK$ is $p$-divisible, $Kv$ is perfect and
 %  $	{\rm gr}(K)$ is perfect and 
  	every simple algebraic extension $(L\mid K, v)$  admits a finite complete sequence of key polynomials (Proposition \ref{teoTameiffGradedRingPerfFiniteSeqKP}).
  	
  	  \vspace{0.3cm}

For a valued field $(K,v)$, consider the associated graded ring ${\rm gr}_v(K)$ of $v$ (see definition in Section \ref{GradedRingsKx}).  Graded rings have shown to be important objects in Valuation Theory and in the study of extension of valuations. For instance, the formula for the defect given in \cite{NartJosneiDefectFormula} and  the works of Teissier connecting Valuation Theory to Toric Geometry  (see \cite{Tei}) use ${\rm gr}_v(K)$.  In  \cite[Remark 7.44]{RamificationValuationCutkosky}, one can see a concrete example where the defect appears as the degree of a field extension, and the fields involved are the field of fractions of graded rings. They are fundamental even for the definition of key polynomials (both Mac Lane-Vaquié conception and Spivakovsky's approach).  In addition, ${\rm gr}_v(K)$ is also related to the study of tame extensions. For example, in \cite{gradedfieldsext} it is shown that if $(K,v)$ is henselian, then there exists an one-to-one correspondence between the field extensions of the  field of fractions of ${\rm gr}_v(K)$  and the tame extensions of $(K,v)$. Other relations can be seen in \cite{BoulagouazTeo52}. Here, we will study the Frobenius endomorphism $F$ on ${\rm gr}_v(K)$ and prove that $F$
is surjective if and only if $v K$ is $p$-divisible  and $Kv$ is perfect (Proposition \ref{lemFsurjectivePDivPerf}).  
Hence, we can rephrase Proposition \ref{teoTameiffGradedRingPerfFiniteSeqKP} in a way that highlights the relation between tame fields and  ${\rm gr}_v(K)$.  

 \vspace{0.3cm}

We now describe the structure of this paper. In Section \ref{ValDef} we define valuations and the defect of a finite extension of valued fields.  Sections \ref{GradedRingsKx} and \ref{Augmentations} consist of a compilation of definitions and results (mostly from \cite{Nart}) on Mac Lane-Vaquié key polynomials, augmentations and Mac Lane-Vaquié chains of valuations. At the end of Section \ref{Augmentations} we present a formula from \cite{NartJosneiDefectFormula} that describes the defect of a simple valued field extension in terms of the defect of each augmentation in an associated chain of valuations.

 \vspace{0.3cm}
 
 Section \ref{CompleteSets} deals with (abstract) key polynomials (in the sense of \cite{spivamahboubkeypoly} and \cite{josneiKeyPolyPropriedades}) and complete sequences. We give proofs for Proposition \ref{lemSeqFinitaComplEPsimCommax}, Theorem \ref{propFiniteSeqiffDefectlessUnibranched} and Corollary \ref{teoHensDefectlessiffKjfiniteCompleteseq} mentioned above.

	\vspace{0.3cm}

 Section \ref{SimplyDefectless} is where we study what we call \textit{simply defectless fields} (i.e. valued fields for which every simple algebraic extension is defectless). 
 We prove that if  $vK$ is $p$-divisible, $Kv$ is perfect
 %${\rm gr}(K)$ is perfect 
 and $(K,v)$ is simply defectless, then it is defectless (Proposition \ref{teoSimplyDefectlessAndDefecteless}). We also relate finite complete sequences of key polynomials to the notion of \textit{algebraically maximal fields} (i.e. fields with no proper algebraic immediate extensions). 

\vspace{0.3cm}

Proposition \ref{teoTameiffGradedRingPerfFiniteSeqKP} mentioned above is proved in Section \ref{TameFields}, where we define tame fields and develop some initial properties of these valued fields. 

\vspace{0.3cm}

We highlight two properties which appear in Corollary \ref{teoHensDefectlessiffKjfiniteCompleteseq} and Corollary \ref{corFCS1simplydefectless}.

\begin{description}
	\item[(\textbf{FCS})] every simple extension $(L\mid K,v)$ admits a finite complete sequence of key polynomials;
	
	\item[(\textbf{FCS*})] every finite extension $(L\mid K, v)$ is such that each simple subextension \linebreak $(K_j\mid K_{j-1},v)$ admits a finite complete sequence of key polynomials.
\end{description}

\vspace{0.3cm }

Suppose $(K,v)$ is henselian and denote also by $v$ the unique extension of the valuation from $K$ to a fixed algebraic closure $\overline{K}$. At the end of Section \ref{TameFields}, we will have a diagram of implications as illustrated in Figure  \ref{figImplications}.

\begin{figure}[H]
	\begin{center}
		\includegraphics[scale=0.27]{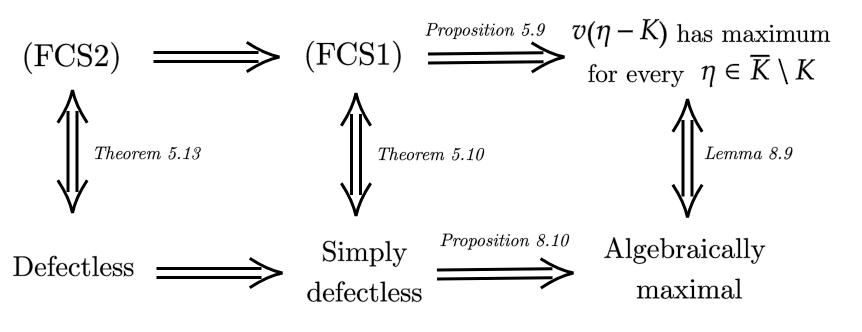}
		\caption{Diagram of implications when $(K,v)$ is henselian. }\label{figImplications}
	\end{center}
\end{figure}

 In Section \ref{GradedRings} we prove Proposition \ref{lemFsurjectivePDivPerf} using  the isomorphism between ${\rm gr}(\VR_K)$ and the semigroup ring $Kv [t^{v K^{\geq 0}}]_\epsilon$ given in \cite{Matheus}.

	\vspace{0.3cm}

\vspace{0.3cm }

In Section \ref{Purely}  we make some considerations on purely inertial and purely ramified extensions. We will study the corresponding \textit{module of Kähler differentials} of these extensions when $vK$ is $p$-divisible or $Kv$ is perfect.
   This module is related to the studies developed recently in \cite{cutkoskykuhlmannKahlerDeeplyRamified}, \cite{cutkoskykuhlmannrzepkaIndependetDefect} and \cite{josneiSpivaKahlerDif} about \textit{deeply ramified fields}. 
 
	\vspace{0.3cm}

%
%
%
%
%\par\medskip
\textbf{Acknowledgments.} The author would like to thank J. Novacoski and M. Spivakovsky for the discussions, suggestions and the careful reading of the first version of this paper.
I also would like to thank the anonymous referee for 
providing useful suggestions  that improved the exposition of the results.

\vspace{0.1cm}

 During the realization of this project the author was supported by a grant from Funda\c c\~ao de Amparo \`a Pesquisa do Estado de S\~ao Paulo (process numbers 2021/13531-0 and 2023/04651-8).

\section{Valuations and defect extensions}\label{ValDef}

%definição valorização, notação

\begin{Def}
Take a commutative ring $R$ with unity. A \index{Valuation}\textbf{valuation} on $R$ is a mapping $\nu:R\lra \Gamma_\infty :=\Gamma \cup\{\infty\}$ where $\Gamma$ is a totally ordered abelian group (and the extension of addition and order to $\infty$ is done in the natural way), with the following properties:
\begin{description}
	\item[(V1)] $\nu(ab)=\nu(a)+\nu(b)$ for all $a,b\in R$.
	\item[(V2)] $\nu(a+b)\geq \min\{\nu(a),\nu(b)\}$ for all $a,b\in R$.
	\item[(V3)] $\nu(1)=0$ and $\nu(0)=\infty$.
\end{description}
\end{Def}

%\vspace{0.2cm} 	

 Let $\nu: R \lra\Gamma_\infty$ be a valuation. The set $\supp(\nu)=\{a\in R\mid \nu(a )=\infty\}$ is a prime ideal, called the \textbf{support} of $\nu$. The \textbf{value group} of $\nu$ is the subgroup of $\Gamma$ generated by 
$\{\nu(a)\mid a \in R\setminus \supp(\nu) \}
$ and is denoted by $\nu R$ or $\Gamma_\nu$. A valuation $\nu$ is a \index{Valuation!Krull}\textbf{Krull valuation} if $\supp(\nu)=(0)$.  If $\nu$ is a Krull valuation, then $R$ is a domain and we can extend $\nu$ to $K={\rm Frac}(R)$ on the usual way. In this case, 
 define the \textbf{valuation ring} as $\VR_{K}:=\{ a\in K\mid \nu(a)\geq 0 \}$. The ring $\VR_K$ is a local ring with unique maximal ideal $\MI_K=\{a\in K\mid \nu(a)>0 \}. $ We define the \textbf{residue field} of $\nu$ to be the field $\VR_K/\MI_K$ and denote it by $ K\nu$. The image of $a\in \VR_K$ in $ K\nu$ is denoted by $a\nu$. 
 
 \vspace{0.3cm}

Let $(K,v)\subset (L,w)$ be a valued field extension. We have that $vK$ can be seen as a subgroup of $wL$ and $Lw$ as a field extension $Kv$. We define 
$$e(w/v):=(wL:vK) \text{ and }  f(w/v):= [Lw: Kv]$$
the \textbf{ramification index} and the \textbf{inertia degree}, respectively. 	If $n=[L:K]$ is finite, then both the ramification index and the inertia degree are finite and $e(w/v)f(w/v)\leq n$. In this situation, we can have only a finite number of extension of $v$ from $K$ to $L$. Denoting by $w_1, \ldots, w_r$ all these extensions, we have the \textit{fundamental inequality} (see \cite{Eng})
$$\sum_{i=1}^r e(w_i/v)f(w_i/v)\leq [L:K].$$

A valued field $(K,v)$ is said to be \textbf{henselian} if for every algebraic extension $L$ of $K$ there exists only one extension of $v$ from $K$ to $L$. Fix an algebraic closure $\overline{K}$ of K and an extension $\overline{v}$ of $v$ to $\overline{K}$. Take the separable closure $K^{\rm sep}$ of $K$ in $\overline{K}$ and consider $K^h$ the fixed field  of $G^d:=\{\sigma\in {\rm Gal}(K^{\rm sep}\mid K)\mid \overline{v}\circ \sigma =\overline{v}\}$.  Taking the restriction $v^h=\overline{v}|_{K^h}$,  the valued field $(K^h,v^h)$ is a \textbf{henselization} of $(K,v)$, that is, a henselian field which is contained in every other henselian field that extends $(K,v)$. All henselizations of $(K,v)$ are isomorphic.

\vspace{0.3cm}

Let $(K,v)\subset (L,w)$ be a valued field extension. Take henselizations $K^h$ and $L^h$ of $K$ and $L$ inside $\overline{K}=\overline{L}$.
% Set $v^h = \overline{v}|_{K^h}$ and
%$w^h = \overline{w}|_{L^h}$. 
 We define the \textbf{defect} of the extension $w/v$ as
$$d(w/v) = \frac{[L^h:K^h]}{e(w/v)f(w/v)}. $$
If we call also by $v$ the valuation on $L$ extending $v$ on $K$, then we will write the extension as $(L\mid K, v)$ and denote the defect of the extension by $d(L\mid K, v)$.

\vspace{0.3cm}

Take $w_1, \ldots, w_r$ all extensions of $v$ to $L$. We have the following equality (see \cite{Eng}):
$$ [L:K]  = \sum_{i=1}^r e(w_i/v)f(w_i/v)d(w_i/v).$$
If the extension $w$ is unique, then 
$$d(w/v) = \frac{[L:K]}{e(w/v)f(w/v)} \text{ and } n=efd.$$
We say that $w/v$ is a \textbf{defect extension} if $d(w/v)>1$ and a \textbf{defectless extension} if $d(w/v)=1.$ The valued field $(K,v)$ is called \textbf{defectless} if every finite valued field extension of $K$ is a defectless extension.

\section{Graded rings and Mac Lane-Vaquié key polynomials}\label{GradedRingsKx}

Let $\nu$  be a valuation on a integral domain $R$.
For each $\gamma\in \nu(R)$, we consider the abelian groups 
%\pagebreak
$$\mathcal{P}_\gamma = \{ a\in R\mid \nu(a)\geq \gamma \} \text{ and }  \mathcal{P}_{\gamma}^{+} = \{ a\in R\mid \nu(a)>\gamma \}. $$
%\pagebreak

\vspace{0.2cm} 

\begin{Def}\label{defAnelGrad}
	The \textbf{graded ring } associated to $\nu$  is defined by 
	$${\rm gr}_{\nu}(R):= \bigoplus_{\gamma\in \nu(R)}\mathcal{P}_\gamma/ \mathcal{P}_{\gamma}^{+}.$$	
\end{Def}

\vspace{0.2cm} 

For simplicity of notation, we will also write $\mathcal{G}_\nu$ (when the ring $R$ is clear) or $ {\rm gr}(R)$ (when the valuation is $\nu$ is clear) to denote ${\rm gr}_{\nu}(R)$. 

\vspace{0.2cm} 

The sum on $\mathcal{G}_\nu$ % ${\rm gr}_{\nu}(R)$
 is defined coordinatewise and the product is given by extending  the product of homogeneous elements, which is described by 
$$\left( a+ \mathcal{P}_{\nu(a)}^{+}  \right) \cdot \left( b+ \mathcal{P}_{\nu(b)}^{+} \right): = \left(  ab +\mathcal{P}_{\nu(a)+\nu(b)}^{+}\right). $$

\vspace{0.2cm} 

For $a\not \in \supp(\nu)$,  we denote by $\inv_\nu(a):=a+ \mathcal{P}_{\nu(a)}^{+}$ the image of  $a$ in $\mathcal{P}_{\nu(a)}/ \mathcal{P}_{\nu(a)}^{+} \subseteq \mathcal{G}_\nu.$  If $a\in \supp(\nu)$, then we define $\inv_\nu(a)=0$. The next lemma follows from the definitions above. 

\vspace{0.2cm} 

\begin{Lema}\label{lemPropinv} 
	Let $a,b\in R$. We have the following. 
	\begin{enumerate}
		
		\item $\mathcal{G}_\nu$ 
		%${\rm gr}_{\nu}(R)$ 
		is an integral domain. 
		
		\item $\inv_\nu(a)\cdot \inv_\nu(b)=\inv_\nu(ab)$.

		\item $\inv_\nu(a)=\inv_\nu(b)$ if and only if 
		%$\nu(f)=\nu(g)$ and 
		$\nu(a-b)>\nu(a)=\nu(b)$.
		
		\item If $\nu(a)=\nu(b)= \nu(a+b)$, then $\inv_\nu(a+b)=\inv_\nu(a)+\inv_\nu(b)$.
	\end{enumerate}

\end{Lema}

\vspace{0.2cm} 

We focus on valuations on the polynomial ring $K[x]$.
Let $\nu_i$ and $\nu_j$ be valuations on $K[x]$ with common value group  such that $\nu_i(f)\leq \nu_j(f)$ for all $f\in K[x]$ (denoted by $\nu_i<\nu_j$). Let $\mathcal{P}_\gamma(K[x],\nu_i) =\{ f\in K[x]\mid \nu_i(f)\geq \gamma \} $ (analogously we define $\mathcal{P}_\gamma(K[x],\nu_{j}), \mathcal{P}_{\gamma}^{+}(K[x],\nu_i)$ and $\mathcal{P}_{\gamma}^{+}(K[x],\nu_{j})$).
We have the inclusions
$$\mathcal{P}_\gamma(K[x],\nu_i)\subseteq  \mathcal{P}_\gamma(K[x],\nu_{j}) \text{ and } \mathcal{P}_{\gamma}^{+}(K[x],\nu_i)\subseteq  \mathcal{P}_{\gamma}^{+}(K[x],\nu_{j}) $$
for any $\gamma\in \nu_i(K[x])= \nu_{j}(K[x])$. 
%For simplicity, denote ${\rm gr}_{\nu_i}(K[x])=\mathcal G_{\nu_i}$.
We consider the following map: %defined in \cite{Andrei}:
\begin{align}\label{eqPhiij}
	\phi_{ij}: \hspace{0.3cm}\mathcal G_{\nu_i}\hspace{0.2cm} &\longrightarrow \hspace{0.2cm} \mathcal G_{\nu_j}\\
	\inv_{\nu_i}(f) &\longmapsto  \begin{cases}
		\inv_{\nu_j}(f)& \mbox{ if }\nu_i(f)=\nu_{j}(f)\\ 
		0&\mbox{ if }\nu_i(f)<\nu_{j}(f),
	\end{cases}  \nonumber
\end{align}

\noindent and we extend this map naturally for an arbitrary element. This map is well-defined \cite[Corollary 5.5]{Andrei} and, by construction, it is a homomorphism of graded rings.

\vspace{0.2cm}

	Take $f,g\in K[x]$.
	We say that $f$ is $\boldsymbol{\nu}$\textbf{-equivalent} to $g$ (denoted by $f\sim_\nu g$) if $\inv_\nu(f) = \inv_\nu(g)$.
		 We say that $g$ $\boldsymbol{\nu}$\textbf{-divides} $f$ (denoted by $g\mid_\nu f$) if there exists \linebreak $h\in K[x]$ such that $f\sim_\nu g\cdot h$.
	The polynomial $f$ is $\boldsymbol{\nu}$\textbf{-minimal} if $f\mid_\nu h$ implies $\deg(h)\geq \deg(f)$  for every $h\in K[x]$. We say that $f$ is $\boldsymbol{\nu}$\textbf{-irreducible} %if $f\mid_\nu g\cdot h$ implies $f\mid_\nu g$ or $f\mid_\nu h$ (i.e.
	 if $\inv_\nu(f)\mathcal{G}_\nu$ is a non-zero prime ideal.

\begin{Def}
	A monic polynomial $Q\in K[x]$ is a \textbf{Mac Lane-Vaquié (MLV) key polynomial} for $\nu$ if it is $\nu$-irreducible and $\nu$-minimal.
\end{Def}

Denote by ${\rm KP}(\nu)$ the set of all MLV key polynomials for $\nu$. For all $\phi\in {\rm KP}(\nu)$ we write
$$[\phi]_\nu := \{\varphi\in {\rm KP}(\nu)\mid \phi \sim_\nu \varphi\}. $$
All polynomials in $[\phi]_\nu$ have the same degree.

\vspace{0.3cm}

For each $f\in K[x]$ and $\phi$ a non-constant polynomial there exist uniquely determined polynomials $f_0,\ldots, f_r\in K[x]$ with $\deg(f_k)<\deg(\phi)$ for every $k$, $0\leq k\leq r$, such that 	$f = f_0+f_1\phi+\ldots +f_r\phi^r.$
We call this expression the \textbf{$\boldsymbol{\phi}$-expansion} of $f$.
If  $\phi\in {\rm KP}(\nu)$, then  by \cite[Proposition 2.3]{NartKeyPolyValuedFields} the $\nu$-minimality of $\phi$ is equivalent to 
$$\nu(f) = \min_{0\leq k \leq r}\{\nu(f_k\phi^k)\}. $$

When ${\rm KP}(\nu)\neq \emptyset$, we define $$\boldsymbol{\deg(\nu)} := \min\{\deg(\phi)\mid \phi \in {\rm KP}(\nu)\}.$$ 
If $\supp(\nu)\neq (0)$, then  ${\rm KP}(\nu)= \emptyset$ (see \cite[Theorem 4.4]{NartKeyPolyValuedFields} or \cite[Theorem 2.2]{NartJosneiDefectFormula}). In this case,
if $\supp(\nu)=(g)$ for some irreducible polynomial $g$, then we define $\deg(\nu):=\deg(g)$.

\section{Augmentations and Mac Lane-Vaquié chains}\label{Augmentations}

The content of this section was mostly taken from \cite{Nart}. 

\subsection{Ordinary augmentations }

Let $\mu$ be a valuation on $K[x]$
% with finite range. 
and $\phi$ be a Mac Lane-Vaquié key polynomial for $\mu$. Let $\Lambda$ be a totally ordered group such that $\Gamma_\mu\subset \Lambda$ and $\gamma\in \Lambda_\infty$ such that $\gamma>\mu(\phi)$. For $f\in K[x]$, write the $\phi$-expansion 
$f = f_0+f_1\phi+\ldots +f_r\phi^r. $

\begin{Def}The valuation
	$$\nu(f) := \min_{0\leq k \leq r }\{ \mu(f_k)+k\gamma \} $$
 is called an  \textbf{ordinary augmentation}  of $\mu$. We denote it by 
	$\nu=[\mu;\phi,\gamma]. $
\end{Def}

See   \cite[Theorem 4.4]{MacLane} or \cite[Theorem 5.1]{josneimonomial}  for a proof that $\nu$ defined above is indeed a valuation.
We have the following properties \cite[Proposition 2.1]{Nart}:
\begin{itemize}
	\item $\mu<\nu$, i.e., $\mu(f)\leq \nu(f)$ for every $f\in K[x]$ and there exists $h\in K[x]$ such that $\mu(h)<\nu(h)$;
	
	\item if $\gamma<\infty$, then $\phi$ is a MLV key polynomial for $\nu$ of minimal degree (in particular $\deg(\nu)=\deg(\phi)$). 
	
\end{itemize}

\subsection{Depth zero valuations}
Let $(K,v)$ be a valued field. Take $a\in K$ and $\gamma \in \Lambda_\infty\supset \Gamma_v$. The map $\mu=\omega_{a,\gamma}$, defined as
$$\mu\left( \sum_{0\leq k\leq r}a_k(x-a)^k\right):= \min_{0\leq k\leq r}\{v(a_k)+k\gamma\}, $$
is called a \textbf{depth zero valuation}. If $\gamma<\infty$, then $x-a$ is a MLV key polynomial for $\mu$ and $\deg(\mu)=1$ (see \cite{Nart} for more details).

\subsection{Continuous families of valuations and MLV limit key polynomials}

Take $\mu$ a valuation on $K[x]$ and $\Lambda\supset\Gamma_\mu$. Suppose $\mu$ non-maximal in the set of all valuations of $K[x]$ with value group contained in $\Lambda$ with partial order  given by $\mu_1\leq \mu_2$ if and only if $\mu_1(f)\leq \mu_2(f)$ for every $f\in K[x]$. 
We say that a family $\mathcal{A} = (\rho_i)_{i\in A}$ of valuations on $K[x]$ 
 is a \textbf{continuous  family of augmentations of} $\mu$ of stable degree $m$ if
\begin{itemize}
		\item $A$ is a totally ordered set without maximal element and $i\mapsto \rho_i$ is an isomorphism of totally ordered sets;
	
	\item $\rho_i = [\mu; \chi_i,\beta_i]$ where $\beta_i\in \Lambda$, $\beta_i>\mu(\chi_i)$ and $\beta_i<\beta_j$ for $i<j$;
	
%		\item the set $\{\deg(\rho_i)\}_{i\in A}$ is stable (i.e. there exists $i_0\in A$ such that $\deg(\rho_i)=\deg(\rho_{i_0})$ for all $i\geq i_0$). We denote by $m=\deg(\mathcal{A})$ this stable degree.
%	
	\item all MLV key polynomials $\chi_i\in {\rm KP}(\mu)$ have the same degree $\boldsymbol{\deg(\mathcal{A})}:=m$, called the stable degree of $\mathcal{A}$;
	
	\item for all $i<j$,  $\chi_j$ is a MLV key polynomial for $\rho_i$,  $\chi_i\nsim_{\rho_i}\chi_j$ and \linebreak $\rho_j = [\rho_i;\chi_j,\beta_j]$.
\end{itemize}

\vspace{0.3cm}

A polynomial $f\in K[x]$ is said to be $\mathcal{A}$\textbf{-stable} if there exists $i_f\in A$ such that $\rho_i(f)=\rho_{i_f}(f)$ for all $i\geq i_f$. We denote this stable value by $\rho_\mathcal{A}(f):=\rho_{i_f}(f)$. Let $m_\infty$ be the minimal degree of an $\mathcal{A}$-unstable polynomial. If all polynomials are $\mathcal{A}$-stable, then we set $m_\infty=\infty$. We say that $\mathcal{A}$ is \textbf{essential} if $\deg(\mathcal{A})<m_\infty<\infty$.

\begin{Def}
	A monic polynomial $\phi\in K[x]$ is a \textbf{MLV limit key polynomial} for $\mathcal{A}$ if it is $\mathcal{A}$-unstable and has the smallest degree among $\mathcal{A}$-unstable polynomials. 
\end{Def}

We denote by ${\rm KP}_\infty(\mathcal{A})$ the set of MLV limit key polynomials for $\mathcal{A}$.
Take $\phi\in {\rm KP}_\infty(\mathcal{A}) $ and $\gamma\in \Lambda_\infty$ such that $\gamma>\rho_i(\phi)$ for all $i\in A$. For $f\in K[x]$, write the $\phi$-expansion 
$f = f_0+f_1\phi+\ldots +f_r\phi^r. $

\begin{Def}
	The valuation 
	$$\nu(f) = \min_{0\leq k\leq r}\{\rho_{\mathcal{A}}(f_k)+k\gamma\}, $$
	 is called a \textbf{limit augmentation} of $\mathcal{A}$. We  denote it by  $\nu = [\mathcal{A}; \phi, \gamma]. $
\end{Def}

See \cite[Theorem 6.6]{dutta}, \cite[Theorem 5.16]{josneimonomial} or \cite[Proposition 1.22]{Vaq} for a proof that $\nu$ defined above is indeed a valuation such that $\rho_i<\nu$ for all $i\in A$.
If $\gamma<\infty$, then $\phi\in {\rm KP}_\infty(\mathcal{A})$ is a MLV key polynomial for $\nu=[\mathcal{A};\phi, \gamma]$ and $\deg(\nu) = \deg(\phi)$.

\vspace{0.3cm}

For $\mathcal{A}=(\rho_i)_{i\in A}$  a continuous family of augmentations of $\mu$,  we say that $\mathcal{A}$ \textbf{is based on} $\mu$ if there exists $i_0=\min(A)$ and $\mu=\rho_{i_0}$. 
We will write $\mu\to \nu$ to describe an augmentation. In this case, either
$$\nu = [\mu;\phi, \gamma], \qquad \phi\in {\rm KP}(\mu), \qquad \gamma>\mu(\phi) $$
or
$$ \nu = [\mathcal{A};\phi, \gamma], \qquad \phi\in {\rm KP}_\infty(\mathcal{A}), \qquad \gamma>\rho_i(\phi) \text{ for all } i\in A,$$
with $\mathcal{A}$ based on $\mu$.

\subsection{Mac Lane-Vaquié Chains}\label{SubSecMCVChain}

Let 
\begin{equation}\label{eqMLVChain}
\mu_0 \overset{\phi_1,\gamma_1} {\longrightarrow}\mu_1 \overset{\phi_2,\gamma_2} {\longrightarrow} \cdots  \overset{} {\longrightarrow} \mu_{t-1} \overset{\phi_t,\gamma_t} {\longrightarrow} \mu_t\cdots
\end{equation}
be a finite or countably infinite chain of augmented valuations where each $\mu_{t+1}$ is of one of the following types:
\begin{itemize}
	\item \textbf{Ordinary augmentation:} $\mu_{t+1} = [\mu_t;\, \phi_{t+1}, \gamma_{t+1}]$ for some $\phi_{t+1}\in {\rm KP}(\mu_t)$.
	
	\item \textbf{Limit augmentation:} $\mu_{t+1} = [\mathcal{A};\, \phi_{t+1}, \gamma_{t+1}]$ for some $\phi_{t+1}\in {\rm KP}_{\infty}(\mathcal{A})$, where $\mathcal{A}$ is an essential continuous family of augmentations of $\mu_t$.
	%(that is, $\mathcal{A} = (\rho_i)_{i\in A}$ with $\rho_i = [\mu_n;\chi_i,\beta_i]$ as in \cite{Nart}, Definition 3.2).
\end{itemize}

 By \cite[p.13]{Nart} we have $\gamma_t=\mu_t(\phi_t)<\gamma_{t+1}$ and $\phi_t$ is a key polynomial for $\mu_t$ of minimal degree.
	% and $\deg(\mu_n)=\deg(\phi_n)\mid \deg()$

\vspace{0.3cm}

	Let $\mu$ and $\nu$ be valuations on $K[x]$ such that $\mu<\nu$. Define $$\Phi_{\mu,\nu}:=\{  \phi\in K[x]\mid \phi \text{ monic with minimal degree such that } \mu(\phi)<\nu(\phi) \}.$$
% be the set of monic polynomials $\phi$ of minimal degree satisfying $\mu(\phi)<\nu(\phi)$. 
We call $\Phi_{\mu,\nu}$ the \textbf{tangent direction} of $\mu$ determined by $\nu$. 
By \cite[Corollary 2.5]{Nart} and \cite[Theorem 1.15]{Vaq}, 
every $\phi \in \Phi_{\mu,\nu}$ is a MLV key polynomial for $\mu$ and $\Phi_{\mu,\nu} = [\phi]_\mu$. Denote by $\deg(\Phi_{\mu,\nu})$ the common degree of the polynomials in $\Phi_{\mu,\nu}$. This set has also the following properties (see \cite{Nart}):

\begin{itemize}
	\item for any $\phi\in \Phi_{\rho,\mu}$ we have $\rho < [\rho; \phi, \mu(\phi)]\leq \mu$;
	
	\item for valuations $\rho <\mu <\nu$, we have $\Phi_{\rho,\mu} = \Phi_{\rho,\nu}$;
	
	\item in (\ref{eqMLVChain}), we have $\deg(\mu_t)=\deg(\phi_t)\mid \deg(\Phi_{\mu_t,\mu_{t+1}})$.

\end{itemize}

If $
\mathcal{A}$ is a continuous family of limit augmentations of $\mu$, then we denote $\Phi_{\mu, \mathcal{A}} := \Phi_{\mu,\rho_i} = [\chi_i]_{\mu}$. This set does not depend on $i$. By \cite{Nart}, for a chain of valuations as in (\ref{eqMLVChain}) we have

 $$\Phi_{\mu_{t},\mu_{t+1}}=\begin{cases}
	[\phi_{t+1}]_{\mu_t}, &\text{ if } \mu_t\to\mu_{t+1} \text{ is ordinary},\\
	\Phi_{\mu_t,\mathcal{A}}=[\chi_{i}]_{\mu_t}, &\text{ if } \mu_t\to\mu_{t+1} \text{ is limit},
\end{cases} $$
and
$$\deg(\Phi_{\mu_{t},\mu_{t+1}})=\begin{cases}
	\deg(\phi_{t+1}), &\text{ if } \mu_t\to\mu_{t+1} \text{ is ordinary},\\
	\deg(\mathcal{A}), &\text{ if } \mu_t\to\mu_{t+1} \text{ is limit}.
\end{cases} $$

\begin{Def}
	A finite or countably infinite chain of augmentations as in (\ref{eqMLVChain}) is a \textbf{Mac Lane-Vaquié (MLV) chain}  if every augmentation step satisfies:
	\begin{itemize}
		\item If $\mu_t\rightarrow \mu_{t+1}$ is ordinary, then $\deg(\mu_t)<\deg(\Phi_{\mu_t,\mu_{t+1}})$.
		
		\item If $\mu_t\rightarrow \mu_{t+1}$ is limit, then $\deg(\mu_t) = \deg(\Phi_{\mu_t,\mu_{t+1}})$ and $\phi_t\not\in \Phi_{\mu_t,\mu_{t+1}}$.

	\end{itemize}
	A Mac Lane-Vaquié chain is \textbf{complete} if the valuation $\mu_0$ has depth zero.
\end{Def}

\begin{Obs}
	A MLV chain can be composed by a mix of ordinary and limit augmentation. Moreover, since we consider chains as in (\ref{eqMLVChain}), if $\mu_t\to \mu_{t+1}$ is limit, then the continuous family of augmentations associated to $\mu_t$ is essential, i.e., the degree $\deg(\phi_{t+1})$ of the limit key polynomial is strictly greater than the stable degree $\deg(\mathcal{A})$.  That is, if $\mu_t\to \mu_{t+1}$ is either ordinary or limit augmentation, then $\deg(\mu_t)=\deg(\phi_t)<\deg(\phi_{t+1})=\deg(\mu_{t+1})$ \cite[p. 14]{Nart}.
\end{Obs}

\begin{Teo}\label{teoNartMacLaneVaquieChains}\cite[Theorem 4.3]{Nart}
	Every valuation $\nu$ on $K[x]$ falls in one of the following cases.
	
	\begin{enumerate}
		\item It is the last valuation of a complete finite Mac Lane-Vaquié chain:
	$$\mu_0 \overset{\phi_1,\gamma_1} {\longrightarrow}\mu_1 \overset{\phi_2,\gamma_2} {\longrightarrow} \cdots  \overset{} {\longrightarrow} \mu_{r-1} \overset{\phi_r,\gamma_r} {\longrightarrow} \mu_r=\nu.$$
	
	\item It is the stable limit of a continuous family $\mathcal{A}=(\rho_i)_{i\in A}$ of augmentations of some valuation $\mu_r$ which falls in case 1:
	$$\mu_0 \overset{\phi_1,\gamma_1} {\longrightarrow}\mu_1 \overset{\phi_2,\gamma_2} {\longrightarrow} \cdots  \overset{} {\longrightarrow} \mu_{r-1} \overset{\phi_r,\gamma_r} {\longrightarrow} \mu_r\overset{(\rho_i)_{i\in A}} {\longrightarrow}\rho_{\mathcal{A}}=\nu,$$
	such that the class $\Phi_{\mu_r,\nu}$ has degree $\deg(\mu_r)$ and $\phi_r\not\in \Phi_{\mu_r,\nu}$.

	\item It is the stable limit of a complete infinite Mac Lane-Vaquié chain:
	$$\mu_0 \overset{\phi_1,\gamma_1} {\longrightarrow}\mu_1 \overset{\phi_2,\gamma_2} {\longrightarrow} \cdots  \overset{} {\longrightarrow} \mu_{t-1} \overset{\phi_t,\gamma_t} {\longrightarrow} \mu_t\overset{} {\longrightarrow}\cdots $$
	\end{enumerate}
	
\end{Teo}

\vspace{0.3cm}

By  \cite[Lemma 4.5]{Nart}, $\nu$ falls in Case (1) of Theorem \ref{teoNartMacLaneVaquieChains}  above if and only if $ {\rm KP}(\nu)\neq \emptyset$ or $\supp(\nu)\neq (0)$.

\vspace{0.3cm}

Take 
$$\mu_0 \overset{\phi_1,\gamma_1} {\longrightarrow}\mu_1 \overset{\phi_2,\gamma_2} {\longrightarrow} \cdots  \overset{} {\longrightarrow} \mu_{t-1} \overset{\phi_t,\gamma_t} {\longrightarrow} \mu_t\cdots$$
a Mac Lane-Vaquié chain of lenght $r\in \N_\infty$. Let us fix the following notations and highlight some properties.
\begin{itemize}
	\item $m_t = \deg(\mu_t)=\deg(\phi_t). $ If $\mu_t \rightarrow \mu_{t+1}$ is ordinary, then  $m_{t+1}=\deg(\Phi_{\mu_{t},\mu_{t+1}}).$
	
	\item $e_t = (\Gamma_{\mu_t}:\Gamma_{\mu_{t-1}}).$
	
	\item $\Delta_t = \Delta_{\mu_t}$, the subring of  zero grade elements of $\mathcal{G}_{\mu_t}$.
	
	\item $\kappa_t = \kappa(\mu_t)$, the algebraic closure of $Kv$ in $\Delta_t$.
	
	\item $f_t = [\kappa_{t+1}:\kappa_t] = \deg(\Phi_{\mu_t,\nu})/(e_tm_t)$ \cite[Lemma 5.2]{Nart}.
\end{itemize}

Consider $v_t = [\mu_t;\phi_t,\infty]$ on $K[x]$ for each $t\geq 0$. We have $\Gamma_{v_t} = \linebreak\{\mu_{t-1}(h)\mid h\neq 0 \text{ and } \deg(h)<\deg(\phi_t)\}=\Gamma_{\mu_{t-1}}$ and $K[x]v_t=\kappa_t$ \cite[Propositions 2.12 and 3.6]{NartKeyPolyValuedFields}. Also
$$e(v_t/v) = e_0\cdots e_{t-1} \text{ and } f(v_t/v) = f_0\cdots f_{t-1}. $$

\subsection{The defect formula} We will end this section with a formula that allows us to calculate the defect of a finite valued field extension by using a chain of augmentations. 

\vspace{0.3cm}

We start with a lemma from which we will derive the notion of defect of an augmentation. We define the $\boldsymbol{\nu}$\textbf{-degree} of $f\in K[x]$, denoted by $\deg_\nu(f)$, as the degree of $\inv_\nu(f)$ as a polynomial in $\inv_\nu(\phi)$ over $\mathcal{G}_\nu^0$ (the subalgebra generated by the set of all homogeneous units of $\mathcal{G}_\nu$) for some $\phi\in {\rm KP}(\nu)$ of minimal degree. See more details in \cite[Remark 16]{spivamahboubkeypoly}, \cite[Proposition 3.5]{NartKeyPolyValuedFields} or \cite[Proposition 4.5]{josneimonomial}.

\begin{Lema} \cite[Lemma 6.1]{NartJosneiDefectFormula}
	Let $\mu\to \nu$ be an augmentation. Let $\Phi_{\mu,\nu}$ be the corresponding tangent direction. For $Q\in \Phi_{\mu,\nu}$ set $\rho_Q = [\mu; Q, \nu(Q)]$. Let $\phi\in K[x]$ be either a MLV key polynomial of minimal degree of $\nu$ or $\supp(\nu)=(\phi)$. Then the positive integer $$d = \min\{ \deg_{\rho_Q}(\phi)\mid Q\in \Phi_{\mu,\nu}\} $$
	is independent of the choice of $\phi$. Moreover, the set of all $Q\in \Phi_{\mu,\nu}$ such that $\deg_{\rho_Q}(\phi)=d$ is cofinal in $\Phi_{\mu,\nu}$.
\end{Lema}

\begin{Def}
	The stable value $d$ is called the \textbf{defect} of the augmentation. We denote it by $d(\mu\to\nu)$.
\end{Def}

Some properties of this notion of defect are listed in the lemma below.

\begin{Lema}\cite[Lemmas 6.3 and 6.4]{NartJosneiDefectFormula}
	\begin{enumerate}
		\item 	If $\mu\to \nu$ is an ordinary augmentation, then $d(\mu\to \nu)=1$. 
		
		\item Suppose $(K,v)$ henselian. For all augmentations $\mu\to \nu$ we have 
		$$d(\mu\to\nu) = \deg(\nu)/\deg(\Phi_{\mu,\nu}). $$
		In particular, if $\mu\to\nu$ is a limit augmentation, then $$d(\mu\to\nu) = \deg(\nu)/\deg(\mu).$$
	\end{enumerate}

\end{Lema}

\begin{Obs}
Suppose $(K,v)$ henselian.	If $\mu\to\nu$ is a limit augmentation which comes from an essential continuous family of augmentations based on $\mu$, then \linebreak $d(\mu\to\nu)=\deg(\nu)/\deg(\mu)>1$. Hence, for a MLV chain, $\mu_t\to\mu_{t+1}$ is ordinary if and only if $d(\mu_t\to\mu_{t+1})=1$. 
\end{Obs}

We say that an augmentation $\mu\to\nu$ is \textbf{proper} when the class $\Phi_{\mu,\nu}$ is proper, that is, when there exists $\varphi\in {\rm KP}(\mu)$ such that $\varphi\nsim_\mu \phi$ for some $\phi\in \Phi_{\mu,\nu}$.
A chain of augmentation is said to be proper if all of its augmentations are proper. By \cite{Nart}, all MLV chains are proper.

\vspace{0.3cm}

The next results state a formula for the defect of a finite simple field extension in terms of the defect of augmentations. 

\begin{Teo}\cite[Theorem 6.14]{NartJosneiDefectFormula}
	Let $w$ be an extension of $v$ to a finite simple extension $L\mid K$. Let $\mu$ be the valuation on $K[x]$ induced by $w$. For any proper chain of augmentation ending on $\mu$ we have 
	$$d(w/v) = d(\mu_0 \rightarrow \mu_1)\cdots d(\mu_r \rightarrow \mu). $$
\end{Teo}

\begin{Cor}\label{CorDefeitoProdGrausNartJosnei}\cite[Corollary 6.16]{NartJosneiDefectFormula}
	Suppose that $(K,v)$ is henselian. Let $w$ be an extension of $v$ to a finite simple extension $L\mid K$. Let $\mu$ be the valuation on $K[x]$ induced by $w$. For any proper chain of augmentation ending on $\mu$ we have $$d(w/ v) = \prod_{t\in J}\frac{\deg(\mu_{t+1})}{\deg(\mu_t)} $$
	where $J$ contains all indices $t$ such that $\mu_t\rightarrow \mu_{t+1}$ is a limit augmentation.
\end{Cor}

\begin{Obs}
	The above corollary is a version of \cite[Corollary 6.1]{Nart}, where we note that it is not necessary to suppose $(K,v)$ henselian. It is sufficient to suppose that the extension $w$ of $v$ is unique. A similar result can be found in \cite[Corollary 2.10]{Vaq3}.
\end{Obs}

\section{Complete sequences of key polynomials}\label{CompleteSets}

\subsection{Key polynomials}

Fix a valuation  $\nu$ on $K[x]$, the ring of polynomials in one indeterminate over the field $K$. Our definition of key polynomial relates to the ones in \cite{josneiKeyPolyPropriedades} and  \cite{spivamahboub}.  Fix an algebraic closure $\overline{K}$  for $K$ and fix  a valuation $\overline{\nu}$ on $\overline{K}[x]$ such that $\overline{\nu}|_{K[x]}=\nu$.

\vspace{0.2cm}

	Let $f\in K[x]$ be a non-zero polynomial.
	
	\begin{itemize}
		\item If $\deg(f)>0$, set 
		$$\delta(f):=\max\{\overline{\nu}(x-a)\mid a\in \overline{K} \text{ and }f(a)=0\}.$$
		
		\item If $\deg(f)=0$, set $\delta(f) = -\infty$.
		
	\end{itemize}

\begin{Obs} According to \cite{josneiKeyPolyMinimalPairs}, $\delta(f)$ does not depend on the choice of the algebraic closure $\overline{K}$ or the extension $\overline{\nu}$ of $\nu$. 
	
\end{Obs}

\begin{Def}\label{defiPoliChave}
	A monic polynomial $Q\in K[x]$ is a  \textbf{(abstract) key polynomial} of level $\delta(Q)$ if, for every $f\in K[x]$,
	$$\delta(f)\geq \delta(Q) \Longrightarrow \deg(f)\geq \deg(Q). $$
\end{Def}

	Let $q\in K[x]$ be a non-constant polynomial and $\nu$ a valuation on $K[x]$. 
	For a given $f\in \K[x]$, denote by $f_0, \ldots, f_r$ the coefficients of the $q$-expansion of $f$. 
	The map 
	$$\nu_q(f):=\underset{0\leq k\leq r}{\min} \{ \nu(f_kq^k)  \},$$
is called the \textbf{truncation} of $\nu$ at $q$.
This map is not always a valuation, as we can see in \cite[Example 2.4]{josneiKeyPolyPropriedades}.   If $Q$ is a key polynomial, then $\nu_Q$ is a valuation on $K[x]$ \cite[Proposition 2.6]{josneiKeyPolyMinimalPairs} and $Q\in {\rm KP}(\mu_Q)$  \cite[Corollary 4.7]{josneimonomial}.

\vspace{0.3cm} 	

In the following lemmas, we state some properties of key polynomials and truncations.

\begin{Lema}\label{prop3itensQQlinhaPoliChaves}\cite[Proposition 2.10]{josneiKeyPolyPropriedades}
	Let $Q,Q'\in K[x]$  be key polynomials for $\nu$.
	%with  $\delta(Q),\delta(Q')\in \Gamma_\Q$.  
	We have the following.
	\begin{enumerate}
		\item If $\deg(Q)<\deg(Q')$, then $\delta(Q)<\delta(Q')$.
		
		\item If $\delta(Q)<\delta(Q')$, then  $\nu_Q(Q')<\nu(Q')$.
		
		\item If $\deg(Q)=\deg(Q')$, then
		$$\nu(Q)<\nu(Q')\Longleftrightarrow \nu_Q(Q')<\nu(Q')\Longleftrightarrow \delta(Q)<\delta(Q'). $$
	\end{enumerate}
\end{Lema}

For a key polynomial $Q$ we consider the following set:
$$\Psi(Q):=\{f\in K[x]\mid f \text{ is  monic with minimal degree such that } \nu_Q(f)<\nu(f)  \}. $$

In the language of Section \ref{Augmentations}, we have $\Psi(Q) = \Phi_{\nu_Q,\nu}$, the tangent direction of $\nu_Q$ determined by $\nu$. 

\begin{Lema}\label{lemaPsiQ} \cite[Lemma 2.11]{josneiKeyPolyPropriedades}
If $Q$ is a key polynomial, then every $Q'\in \Psi(Q)$ is also a key polynomial. Moreover, $\delta(Q)<\delta(Q')$.
\end{Lema}

\subsection{Complete sequences of key polynomials}

\begin{Def}

	A set $\boldsymbol{Q}\subset K[x]$ is called a \textbf{complete set} for $\nu$ if for every $f\in \K[x]\setminus\{0\}$ there exists $q\in \boldsymbol{Q}$ such that 
	$$\deg(q)\leq \deg(f) \text{ and } \nu(f)=\nu_q(f). $$
	
	A set $\boldsymbol{Q} = \{Q_i\}_{i\in I}$  is said a \textbf{sequence of key polynomials} if every element $Q_i$ is a key polynomial, the set $I$ is well ordered and the map $i\to Q_i$ is order preserving (with respect to $\delta$). A sequence of key polynomials $\boldsymbol{Q} = \{Q_i\}_{i\in I}$ is a \textbf{complete sequence of key polynomials}  if the set $\boldsymbol{Q}$ is a complete set for $\nu$.
	
	\end{Def}

According to \cite[Theorem 1.1]{josneiKeyPolyPropriedades}, %and Theorem 3.17 of \cite{josneimonomial}
every valuation $\nu$ on $K[x]$ admits a complete sequence $\boldsymbol{Q}$ of key polynomials. Moreover, the next proposition gives us more details about $\boldsymbol{Q}$.

\begin{Prop}\label{propAdmissibleFamilyJosneiTrunc}
	\cite[Corollary 6.5]{josneimonomial} Let $\nu$ be a valuation on $K[x]$, set $\nu_0:=\nu|_K$. Then there exists a family $$ \mathcal F = \{(\nu_Q, Q, \nu(Q))  \}_ {Q\in \boldsymbol{Q}}$$
	such that for every $f\in K[x]$, there exists $Q\in \boldsymbol{Q}$ such that $\nu(f) = \nu_{Q'}(f)$ for every $Q'\in \boldsymbol{Q}$ with $\delta(Q)\leq \delta(Q')$ (in particular, $\boldsymbol{Q}$ is complete). The set $\boldsymbol{Q}$ can be chosen as 
	$$\boldsymbol{Q} = \bigcup_{i=1}^N I_i, $$
	where $N\in \N\cup \{\infty\}$, with the following properties.
	% and $\mathcal F$ is a continuous family of augmentations ending in $\nu$. 
%	Explicitly, we have the following properties. 
	\begin{enumerate}
		\item $\boldsymbol{Q}$ admits a smallest element $Q_0$ (with respect to $\delta$) of the form $x-a\in K[x]$ and 
		$$\nu_{Q_0} = [\nu_0;\, \nu_{Q_0}(Q_0)=\nu(Q_0)]. $$
		
		\item For every $i$, $1\leq i <N$, we can write $I_i = A_i\cup B_i$ where
		\begin{itemize}
			\item $B_i = \{Q_{i,1}, \ldots, Q_{i,n_i}\}$ for some $n_i\in \N$ such that 
			$$\nu_{Q_{i,j+1}} = [\nu_{Q_{i,j}};\, \nu_{Q_{i,j+1}}(Q_{i,j+1})=\nu(Q_{i,j+1})] $$
			for every $j$, $1\leq j <n_i$, and 
			$$\deg(Q_{i,1})<\deg(Q_{i,2}) <\cdots < \deg(Q_{i, n_i});$$
			
			\item $\mathcal{A}_i=\{\nu_{Q}\}_{Q\in A_i}$ is an essential continuous family of augmentations based on $\nu_{Q_{i,n_i}}$ with $\deg(Q)>\deg(Q_{i,n_i})$ for every $Q\in A_i$. 
			
		\end{itemize}
	
	\item For every $i$, $1<i<N$, $I_i$ admits a first element $Q$ which is a MLV-limit key polynomial for $\mathcal{A}_{i-1}$.
	
	\item If $N<\infty$, then $I_N = A_N\cup B_N$ with
	\begin{itemize}
		\item $A_N$ and $B_N$ as in $(3)$, or
		
		\item $A_N = \emptyset$ and $B_N = \{Q_{N,1}, \ldots, Q_{N,L-1}\}$ for some $L\in \N\cup \{\infty\}$ such that 
	$$\nu_{Q_{i,j+1}} = [\nu_{Q_{i,j}};\, \nu_{Q_{i,j+1}}(Q_{i,j+1})=\nu(Q_{i,j+1})] $$
	for every $j$, $1\leq j <L$ (if $B_N$ is finite, consider $L = |B_N|+1$).
	\end{itemize}
	\end{enumerate}
\end{Prop}

\begin{Obs}
	In the language of \cite{Nart}, from the family $\mathcal{F}$ we can construct a Mac Lane-Vaquié chain ending in $\nu$.
\end{Obs}

\subsection{Finite complete sequences of key polynomials}

	Let $(L\mid K, v)$ be a simple algebraic extension and take $\eta$ a generator for $L\mid K$.   Let $\nu$ be the valuation on $K[x]$  defined by $\eta$ and $v$: for $f\in K[x]$, we define
	\begin{equation}\label{eqInducedVal}
		\nu(f):=v(f(\eta)). 
	\end{equation}

	 Take $g$ the minimal polynomial of $\eta$ over $K$ with degree $n$. Since $\nu(g) = \infty$, we have also $\delta(g)=\infty$ and $\supp(\nu)=(g)$. Hence $g$ is a key polynomial for $\nu$ (if $\deg(f)<\deg(g)$, then $f\not\in \supp(\nu)$, that is, $\delta(f)<\infty$).  Moreover, there is no key polynomial of degree greater than $n$ (since a key polynomial of degree greater than $n$ must have $\delta$ greater than $\delta(g)=\infty$, what does not happen).
	 
	 \vspace{0.3cm}

	 We denote by $\Psi_m$ the set of all key polynomials for  $\nu$ of degree $m\geq 1$. We will say that $\Psi_m$ \textit{admits a maximum} (with respect to $\nu$) if $\nu(\Psi_m):=\{\nu(Q)\mid Q\in \Psi_m\}$ admits a maximal element.
	 
	 \vspace{0.3cm}

	  The following proposition is the key that allows us to prove  our main result. It says that a finite complete sequence of key polynomials exists if and only if there are no \textit{limit key polynomials} for $\nu$ (in the sense of \cite{josneiKeyPolyPropriedades}).

\begin{Prop}\label{lemSeqFinitaComplEPsimCommax}
	Let $(L\mid K, v)$ be a simple algebraic extension
%	, take $\eta$ a generator for $L\mid K$ 
	and
 let $\nu$ be the valuation on $K[x]$ defined in (\ref{eqInducedVal}).
  %defined by $\eta$ and $v$.
   Then $\nu$ admits a finite complete sequence of key polynomials if and only if for every $m\geq 1$ either $\Psi_m$ is empty or it admits a maximum. 
\end{Prop}

\begin{proof}Suppose that for every $m$ either $\Psi_m$ is empty or it admits a maximum. For each $m$ such that $\Psi_m\neq \emptyset$, choose one $Q_m$ of maximal value. Let $\boldsymbol{Q}$ be the set of these maximal elements  $Q_m$. 
	%We relabel them in a way that $Q_i\leq Q_j \Leftrightarrow \delta(Q_i)\leq \delta(Q_j)$. 
	We will show that $\boldsymbol{Q}$ is a finite complete sequence of key polynomials. 
	
	\vspace{0.1cm}
	
	Take $f\in K[x]$ any monic polynomial. If $\deg(f)\geq n$, then via the $g$-expansion of $f$ we have
	$$ \nu(f) = \nu(f_0+f_1g+\ldots+f_sg^s)=\nu(f_0)=\nu_g(f)$$
	because $\nu(g)=\infty$. 
	
	\vspace{0.1cm}

		Suppose $\deg(f)<n$. Let $m$ be the biggest integer such that $\Psi_{m}\neq \emptyset$ and $m\leq \deg(f)$.  Take $Q_{m}\in \boldsymbol{Q}$ the  chosen element of maximal value of $ \Psi_{m}$. Let us prove that $\nu_{Q_{m}}(f)=\nu(f)$. Indeed, if $\nu_{Q_{m}}(f)<\nu(f)$, then take any $h\in \Psi(Q_{m})$, which is a key polynomial (Lemma \ref{lemaPsiQ}). We must have $m\leq \deg(h)\leq \deg(f)$. If $m<\deg(h)$, then we contradict the choice of $m$. If $m=\deg(h)$, then by Lemma \ref{lemaPsiQ} and Lemma \ref{prop3itensQQlinhaPoliChaves} (3) we conclude that $\nu(Q_{m})<\nu(h)$, a contradiction since $Q_m$ has maximal value among the elements of $\Psi_m$.  Hence $\nu_{Q_{m}}(f)=\nu(f)$.	
	Therefore, $\boldsymbol{Q}$  is complete. It is finite since $\Psi_m = \emptyset$ for every $i>n$.
	
	\vspace{0.1cm}
	
	 For the converse, let $\boldsymbol{Q} = \{Q_1, \ldots, Q_r\}$ be a finite complete sequence of key polynomials. Take $m$ and suppose $\Psi_m\neq \emptyset$. Take $Q\in \Psi_m$. Hence, there exists $Q_i\in \boldsymbol{Q}$ such that $\deg(Q_i)\leq m$ and $\nu_{Q_i}(Q)=\nu(Q)$. Let us see that $Q_i\in \Psi_m$. Indeed, if $\deg(Q_i)< m = \deg(Q)$, then $\delta(Q_i)<\delta(Q)$ and this implies that $\nu_{Q_i}(Q)<\nu(Q)$ (Lemma \ref{prop3itensQQlinhaPoliChaves}), a contradiction. 
	
	\vspace{0.1cm}
	
	Consider $Q_{\tilde{i}}\in \boldsymbol{Q}$ such that $$\delta(Q_{\tilde{i}})=\max_{1\leq i \leq r}\{ \delta(Q_i)\mid Q_i\in \Psi_m  \}.$$
	% with biggest $\delta$ among the  elements $Q_i$ of the complete set satisfying $Q_i\in \Psi_m$.
	 We will prove that $Q_{\tilde{i}}$ is an element with maximal value in $\Psi_m$. Suppose there is $Q\in\Psi_m$ such that $\nu(Q_{\tilde{i}})<\nu(Q)$. %Hence, $\nu_{Q_{\tilde{i}}}(Q)<\nu(Q)$.
	  However, we know that there exists $Q_j\in \boldsymbol{Q}$ such that $\deg(Q_j)\leq m$ and $\nu_{Q_j}(Q)=\nu(Q)$. By the same reasoning above,  we have $Q_j\in \Psi_m$. 
	 % Then $\nu_{Q_{\tilde{i}}}(Q)<\nu_{Q_j}(Q)$ and Lemma \ref{lempolichaveEpsilonmenor} (2) implies that
	 By Lemma \ref{prop3itensQQlinhaPoliChaves} (3), $\nu(Q_{\tilde{i}})<\nu(Q)$ implies $\delta(Q_{\tilde{i}})<\delta(Q_j)$, contradicting the choice of $Q_{\tilde{i}}$. Therefore, $\Psi_m$ has a maximum. 
\end{proof}

We will call the valued field extension $(L\mid K, v)$ \textbf{unibranched} if  $v$  admits a unique extension from $K$ to $L$.  In the sequence, we prove our main result.

\begin{Teo}\label{propFiniteSeqiffDefectlessUnibranched}
	Let $(L\mid K, v)$ be a simple algebraic extension
%	, take $\eta$ a generator for $L\mid K$ 
and
 let $\nu$ be the valuation on $K[x]$ defined in (\ref{eqInducedVal}).
%defined by $\eta$ and $v$.
 Then $\nu$ admits a finite complete sequence of key polynomials if and only if $d(L\mid K, v)=1$ and $(L\mid K, v)$ is unibranched.
\end{Teo}

\begin{proof}
		 Suppose $d(L\mid K, v)=1$ and $(L\mid K, v)$ unibranched.  Consider the Mac Lane-Vaquié chain ending in $\nu$ given by $ \mathcal F = \{(\nu_Q, Q, \nu(Q))  \}_ {Q\in \boldsymbol{Q}}$  from Proposition \ref{propAdmissibleFamilyJosneiTrunc}.  Corollary \ref{CorDefeitoProdGrausNartJosnei}
		 % and  \cite[Corollary 2.10]{Vaq3} 
		 shows us  that we cannot have limit augmentations in $\mathcal{F}$ (because in a Mac Lane-Vaquié chain, a limit augmentation $\mu_t\to\mu_{t+1}$ satisfies $d(\mu_t\to\mu_{t+1})=\deg(\mu_{t+1})/\deg(\mu_{t})>1$). That is,  $N=1$ in Proposition \ref{propAdmissibleFamilyJosneiTrunc}.
	 This means that the complete set $\boldsymbol{Q}$ from Proposition \ref{propAdmissibleFamilyJosneiTrunc}  is such that $\boldsymbol{Q} = I_1=B_1=\{Q_{1,1}, \ldots, Q_{1,n_1}\}$ for some $n_1\in \N$. Hence, $\boldsymbol{Q}$ is a finite complete sequence of key polynomials.
	 
	 \vspace{0.1cm}
	
	For the converse, consider $g$ the minimal polynomial of $\eta$ over $K$. Suppose $\nu$ admits a finite complete sequence of key polynomials, that is, for all $m$ either $\Psi_m=\emptyset$ or it admits a maximum. Consider $n_0<n_1<\ldots <n_r=\deg(g)$ the indexes such that $\Psi_{n_t}\neq \emptyset$ and $\Psi_m=\emptyset$ for all $n_t<m<n_{t+1}$, where $t\in \{0, \ldots, r-1\}$. Denote by $Q_{n_t}$ a chosen element of $\Psi_{n_t}$,  where $Q_{n_{r}}=g$. Then we have the following chain of augmentations
	$$\mu_0 \rightarrow \mu_1 \rightarrow \cdots \rightarrow \mu_r =\nu$$
	where $\mu_t = \nu_{Q_{n_t}}$.
	This is a finite depth Mac Lane-Vaquié chain and every augmentation is ordinary. Indeed, by Lemma \ref{prop3itensQQlinhaPoliChaves} (2) we have $\mu_{t+1}(Q_{n_{t+1}})<\nu(Q_{n_{t+1}})$ and, since $\Psi_m=\emptyset$ for $n_{t}<m<n_{t+1}$, we conclude that   $Q_{n_{t+1}}\in \Psi(Q_{n_t})$. By Theorem 6.1 of \cite{josneimonomial}, $\mu_{t+1} = [\mu_t; Q_{n_{t+1}}, \nu(Q_{n_{t+1}}) ]$. Since $\deg(\Phi_{\mu_{t+1},\nu})=\deg(Q_{n_{t+1}})=m_t$ (because all augmentations are ordinary) and $\deg(\mu_t)=\deg(Q_{n_t})$, it follows that we have a Mac Lane-Vaquié chain.

	Since 
	%$m_i = \deg(\Phi_{\mu_i,\nu})$ for all $i$ (because all augmentations are ordinary) and
	 $m_{r}=\deg(Q_{n_{r}})=\deg(g)$, the properties listed at the end of Subsection \ref{SubSecMCVChain} lead us to conclude that
	 {\allowdisplaybreaks
	\begin{align*}
		f(L\mid K,v)&=f(\nu/v)\\& =\prod_{t=0}^{r-1}f_t \\
		& = \prod_{t=0}^{r-1}\frac{\deg(\Phi_{\mu_t,\mu_{t+1}})}{e_tm_t}\\
		& = \frac{m_{r}}{e_{r-1}m_{r-1}} \cdots \frac{m_{1}}{e_{0}m_{0}}\\
		& = \frac{m_{r}}{e_{r-1}\cdots e_0}\\
		& =\frac{\deg(g)}{e(L\mid K, v)}.
	\end{align*}}
By the fundamental inequality, we must have $d(L\mid K, v)=1$ and $(L\mid K, v)$ unibranched.
	\end{proof}

\begin{Obs}\label{ObsNaoDependeGerador}
	In (\ref{eqInducedVal}), we see that the definition of $\nu$ depends on the chosen generator $\eta$ for the simple extension. However, since $d(L\mid K,v)$ and the property of being unibranched  do not depend on $\eta$, we conclude that  the existence of a finite complete sequence of key polynomials also does not depend on the choice of the generator of the extension. 
\end{Obs}

The next corollary will give us a characterization of henselian defectless fields in terms of finite complete sequences of key polynomials. We will need the following lemma.

\begin{Lema}\label{lemHensiffEvrySimpleExtUnibranched}
	We have $(K,v)$ henselian if and only if for every simple field extension $L\mid K$ we have $(L\mid K, v)$ unibranched.
\end{Lema}

\begin{proof}
	One direction follows from the definition of a henselian valued field. For the other implication, take $\overline{\mu}_1$ and $\overline{\mu}_2$ extensions of $v$ to $\overline{K}$. For each  $a\in \overline{K}$, consider the restrictions
	$(K(a), \mu_1)$ and $(K(a), \mu_2)$ where $\mu_i=\overline{\mu}_i|_{K(a)}$. By hypothesis, $\mu_1=\mu_2$ in	$K(a)$, hence $\overline{\mu}_1(a)=\overline{\mu}_2(a)$ for every $a\in \overline{K}$. It follows that the extension of $v$ from $K$ to $\overline{K}$ is unique, that is,  $(K,v)$ is henselian. 
\end{proof}

For a  finite extension $(L\mid K, v)$, with $L=K(a_1,\ldots, a_m)$, we write $K_j = K(a_1, \ldots, a_j)$, $1\leq j\leq m$ and consider the simple extension $(K_j\mid K_{j-1},v)$ (with $K_0=K$).
We will say that $(K_j\mid K_{j-1},v)$ \textit{admits a finite complete sequence of key polynomials} if the valuation $\nu_j$ induced by $(K_j\mid K_{j-1},v)$ on the polynomial ring $K_{j-1}[x]$ admits a finite complete sequence of key polynomials.  By Remark \ref{ObsNaoDependeGerador},  this property does not depend on the choice of the generators $a_1,\ldots, a_m$. 

\vspace{0.3cm}

Using the above notation, let us consider the following property.

\begin{description}
%	\item[(\textbf{FCS1})] every simple extension $(L\mid K,v)$ admits a finite complete sequence of key polynomials;
	
	\item[(\textbf{FCS*})] every finite extension $(L\mid K, v)$ is such that each simple subextension \linebreak $(K_j\mid K_{j-1},v)$ admits a finite complete sequence of key polynomials.
\end{description}

\begin{Cor}\label{teoHensDefectlessiffKjfiniteCompleteseq}
	We have $(K,v)$ henselian defectless if and only if $(K,v)$ satisfies (FCS*).
	% for every finite extension $(L\mid K, v)$ each simple subextension $(K_j\mid K_{j-1},v)$ admits a finite complete sequence of key polynomials.
\end{Cor}

\begin{proof}It follows from Theorem \ref{propFiniteSeqiffDefectlessUnibranched}, Lemma \ref{lemHensiffEvrySimpleExtUnibranched} and the fact that for a finite extension $(L\mid K, v)$ we have $$d(L\mid K,v) = \prod_{j=1}^m d(K_j\mid K_{j-1},v).$$
	\end{proof}

\section{Simply defectless  and algebraically maximal valued fields}\label{SimplyDefectless}

\subsection{Simply defectless valued fields}

We will call a valued field $(K,v)$ \textbf{simply defectless} if all simple algebraic extensions of $K$ are defectless.

\vspace{0.3cm}

Every defectless field is simply defectless,
but the converse is not true in general, as it is shown in the example presented in \cite[Section 5]{AghighKhandujaExSimplyDefectlessNotDefectless}.
 Next, we will see that if $vK$ is $p$-divisible and $Kv$ is perfect, %${\rm gr}(K)$ is perfect,
  then we have the converse. 

\begin{Prop}\label{teoSimplyDefectlessAndDefecteless}
	Suppose $(K,v)$ henselian. If $(K,v)$ is simply defectless, $vK$ is $p$-divisible and $Kv$ is perfect,	% and ${\rm gr}(K)$ is perfect
 then $(K,v)$ is a defectless field. 
\end{Prop}

\begin{proof}Suppose $(L\mid K, v)$ is a finite defect extension.  We will show that there exists a simple algebraic extension of $K$ which is not defectless. Take $N$ the normal closure of $L\mid K$. Then $(N\mid K, v)$ is also a finite defect extension. Take $G={\rm Aut}_K(N)$ and consider the fixed field $K^G$, which is a purely inseparable extension of $K$.
	 \vspace{0.1cm}

	We have $(K^G\mid K,v)$ immediate. Indeed, $Kv\subset K^Gv$ is also purely inseparable and $vK^G/vK$ is a $p$-group. Since $Kv$ is perfect and $vK$ is $p$-divisible, it follows that $Kv=K^Gv$ and  $vK=vK^G$. 
	
	 \vspace{0.1cm}
	
	 If $K=K^G$, since $K=K^G\subset N$ is separable and finite, then $(N\mid K,v)$ is a simple defect extension and we are done. 
	If $K^G\neq K$, take $a\in K^G\setminus K$ and consider the simple extension $K\subset K(a)\subset K^G$. This is also an immediate extension, hence $d(K(a)\mid K, v)=[K(a):K]>1$. 
	\end{proof}

 The next corollary follows  from   Theorem \ref{propFiniteSeqiffDefectlessUnibranched} and Lemma \ref{lemHensiffEvrySimpleExtUnibranched}.

\begin{Cor}\label{corFCS1simplydefectless} We have $(K,v)$
	simply defectless and henselian if and only if every simple field extension of $K$ admits a finite complete sequence of key polynomials.
\end{Cor}

We will also name the property appearing in
% Corollary \ref{teoHensDefectlessiffKjfiniteCompleteseq} and 
Corollary \ref{corFCS1simplydefectless}.

\begin{description}
	\item[(\textbf{FCS})] every simple extension $(L\mid K,v)$ admits a finite complete sequence of key polynomials.
	
\end{description}

\vspace{0.3cm}

We always have that (FCS*) implies (FCS). The next corollary states a converse for this implication.

\begin{Cor}\label{lemaFCS1andFCS2}
	Suppose that $vK$ is $p$-divisible and $Kv$ is perfect.
	% $	{\rm gr}(K)$ is perfect.
	 Then (FCS) and (FCS*) are  equivalent.
\end{Cor}

\begin{proof} We only need to prove the converse. It follows from  Corollary \ref{teoHensDefectlessiffKjfiniteCompleteseq} and Proposition \ref{teoSimplyDefectlessAndDefecteless}, since (FCS) is equivalent to $(K,v)$ being henselian and simply defectless by Corollary \ref{corFCS1simplydefectless}.
\end{proof}

\subsection{Algebraically maximal fields} A valued field extension $( L\mid  K, v)$ is said to be \textbf{immediate} if $v  L = v  K$ and $ L v =  K v$. A valued field $(K,v)$ is called \linebreak \textbf{algebraically maximal} if it does not admit proper immediate algebraic extensions. Every henselian defectless field is algebraically maximal, but the converse does not hold in general \cite[p.13]{kuhlmannTameFields}.

\vspace{0.3cm}

The property of being algebraically maximal is related to (FCS), as we see in the following proposition.

\begin{Prop}\label{propFCS1impliesAlgMax}
	If $(K,v)$ satisfies (FCS) (or equivalently if $(K,v)$ is henselian and simply defectless), then it is algebraically maximal. 
\end{Prop}

\begin{proof}
	%We will use the fact that (FCS) is equivalent to $(K,v)$ being henselian and simply defectless (Corollary \ref{corFCS1simplydefectless}). 
	By the contrapositive, suppose $(L\mid K,v)$ is a proper immediate algebraic extension. Take any $\eta\in L\setminus K$ and consider the simple algebraic extension \linebreak $(K(\eta)\mid K, v)$. It is also a proper immediate extension. Since $(K,v)$ is henselian, it is a defect extension. Hence, $(K,v)$ is not simply defectless.

\end{proof}

\begin{Obs}
	Let $\overline{v}$ be a fixed extension of $v$ from $K$ to a fixed algebraic closure $\overline{K}$. We can also prove Proposition \ref{propFCS1impliesAlgMax} using  \cite[Theorem 1.1]{AlgebMaxSimplyDefectAnujSudesh}, which says that	a henselian valued field $(K,v)$ is algebraically maximal if and only if  the set \linebreak $\overline{v}(\eta-K):=\{\overline{v}(\eta-a)\mid a\in K\}$ has a maximum for every $\eta\in \overline{K}\setminus K$. Indeed, 	by Proposition \ref{lemSeqFinitaComplEPsimCommax} and Corollary \ref{corFCS1simplydefectless}, (FCS) implies that $(K,v)$ is henselian and every simple extension $L=K(\eta)$ is such that the set $\overline{v}(\eta -K) = \nu(x-K)=\nu(\Psi_1)$ admits a maximum.

\end{Obs}

\section{Tame fields}\label{TameFields}

\subsection{Tame fields}

A unibranched extension $(L\mid K, v)$ is called \textbf{tame} if every finite subextension $E\mid K$ of $L\mid K$ satisfies the following conditions.

\begin{description}
	\item[(\textbf{TE1})] $(v E:v K)$ is not divisible by $\Char Kv$.
	
	\item[(\textbf{TE2})] $E v\mid Kv$ is a separable extension.
	
	\item[(\textbf{TE3})] $(E\mid K, v )$ is defectless.
\end{description}

\begin{Def}A valued field $(K,v)$ is called a \textbf{tame field} if  it is henselian and its algebraic closure with the unique extension of the valuation is a tame extension. 
	
\end{Def}

All henselian valued fields with $\Char Kv =0$ are tame \cite[p. 14]{kuhlmannTameFields}. Then we will focus on valued fields with positive residue characteristic.  

\vspace{0.3cm}

The following lemma 
can be deduced from the proof of \cite[Theorem 3.2]{kuhlmannTameFields} and further simple arguments.
%)appears partially in the proof of  \cite[Theorem 3.2]{kuhlmannTameFields}. 
%For the completeness of the text, we will present a whole proof.  
%
\begin{Lema}\label{lemTE1pDivTE2perfect}
	Let $(K, v)$ be a henselian valued field with $\Char Kv = p>0$. Take $\overline{K}$ an algebraic closure of $K$.   We have the following.
	
	\begin{enumerate}
		\item $(\overline{K}\mid K, v)$ satisfies (TE1)
		if and only if $v K$ is $p$-divisible.

		\item $(\overline{K}\mid K, v)$ satisfies (TE2) if and only if $Kv$ is perfect.
	\end{enumerate}
	
\end{Lema}

Let us present two characterizations of tame fields found in the literature. The first one is from  \cite{kuhlmannTameFields}.

\begin{Prop}\label{propKulhmannAlgMaxvKDivKvPerfect} \cite[Theorem 3.2]{kuhlmannTameFields} A valued field $(K,v)$ is tame if and only if it is algebraically maximal, $vK$ is $p$-divisible and $Kv$ is perfect.
\end{Prop}

In \cite{BoulagouazTeo52}, we find a description of a tame extension  in terms of the associated graded ring ${\rm gr }(K)$ of $v$. Denote by ${\rm Frac}({\rm gr }(K))$ the  field  of fractions of ${\rm gr }(K)$.

\begin{Prop}\label{propBoulagouazTeo5}
	\cite[Theorem 5]{BoulagouazTeo52} Let $(L\mid K, v)$ be a finite extension of degree $n$. Then $(L\mid K, v)$ is a tame extension  if and only if ${\rm Frac}({\rm gr }(L))$ is a separable extension of ${\rm Frac}({\rm gr }(K))$ of degree $n$.
\end{Prop}

\begin{Obs}\label{obsBoulagouazRPerfect}
	In  \cite{GradedFieldsGaloisTheoryBoulagouaz}, it is developed  the theory of algebraic graded field extensions. It is proved that if $R=\bigoplus_{\gamma\in \Gamma}R_\gamma$ is a graded field of characteristic $p>0$ over an abelian torsion-free grading group $\Gamma$, then $R$ is perfect (which  in \cite{GradedFieldsGaloisTheoryBoulagouaz} means that every \textit{graded field extension} of $R$ is separable) if and only if $R_0$ is perfect and $\Gamma$ is $p$-divisible.  
	%	Hence, for our particular graded field ${\rm gr}(K)$, Proposition \ref{lemFsurjectivePDivPerf} provides a way of seeing that the surjectiveness of the Frobenius endomorphism is equivalent to the definition of perfect graded field given in \cite{GradedFieldsGaloisTheoryBoulagouaz}. 
	Hence, by Proposition \ref{propBoulagouazTeo5} above, $(K,v)$ is a tame field if and only if ${\rm gr}(K)$ is perfect (in the terms of \cite{GradedFieldsGaloisTheoryBoulagouaz}) and for every finite extension $(L\mid K,v)$ we have $[{\rm Frac}({\rm gr }(L)):{\rm Frac}({\rm gr }(K))]=[L:K]$.
\end{Obs}

In the next section  we will show a way of characterizing tame fields in terms of complete sequences of key polynomials.

\subsection{Tame fields and finite complete sequences of key polynomials } $\,$

\vspace{0.3cm}

During this section we suppose  $\Char Kv = p>0$. Although the next result can be proved using Proposition \ref{propKulhmannAlgMaxvKDivKvPerfect}, we will first present a proof using  the results around the property (FCS) studied in this paper.

\begin{Prop}\label{teoTameiffGradedRingPerfFiniteSeqKP}
	Let $(K, v)$ be a valued field. Then $(K,v)$ is a tame field if and only if $(K,v)$ satisfies (FCS), $vK$ is $p$-divisible and  $Kv$ is perfect.
	%.$	{\rm gr}(K)$ is perfect 
	
\end{Prop}

\begin{proof}If $(K,v)$ is tame, then it is henselian and defectless by definition and
	Lemma \ref{lemTE1pDivTE2perfect} says that $vK$ is $p$-divisible and $Kv$ is perfect. In particular, $(K,v)$ is simply defectless. Hence, by Corollary \ref{corFCS1simplydefectless} we have (FCS). 
	\vspace{0.1cm}
	
	For the converse,  $vK$ $p$-divisible, $Kv$ perfect and (FCS) together imply (FCS*) by Corollary \ref{lemaFCS1andFCS2},  which is equivalent to $(K,v)$ being henselian and defectless by  Corollary \ref{teoHensDefectlessiffKjfiniteCompleteseq}. By Lemma \ref{lemTE1pDivTE2perfect}, we have (TE1) and (TE2).

\end{proof}

\begin{Obs}
One can also deduce the converse of Proposition \ref{teoTameiffGradedRingPerfFiniteSeqKP} using algebraically maximal fields, by noticing that (FCS) implies $(K,v)$ algebraically maximal (Proposition \ref{propFCS1impliesAlgMax}). Then $(K,v)$ is a tame field by Proposition \ref{propKulhmannAlgMaxvKDivKvPerfect}.  
\end{Obs}

When $(K,v)$ is henselian, we have the diagram of implications as illustrated in Figure \ref{figImplications} at the Introduction.
We end this section with a compilation of equivalences proved or stated in this paper.

\begin{Cor}
	Suppose $(K,v)$ henselian, $vK$ is $p$-divisible and $Kv$ is perfect. The following  properties are equivalents. 
\begin{enumerate}[label=\rm (\roman*)]
	
	\item $(K,v)$ is  defectless.
	
	\item $(K,v)$ is  simply defectless.
	
	\item $(K,v)$ satisfies (FCS).

	\item $(K,v)$ satisfies (FCS*).
	
	\item $(K,v)$ is tame.

	\item $(K,v)$ is  algebraically maximal.
	
	\item $(K,v)$ is such that $\overline{v}(\eta-K)$ has a maximum for every $\eta\in \overline{K}\setminus K$.
	
	\item $(K,v)$ is such that every finite extension $(L\mid K,v)$ satisfies $[{\rm Frac}({\rm gr }(L)):{\rm Frac}({\rm gr }(K))]=[L:K]$.
\end{enumerate}
\end{Cor}

In the next two last sections, we will study the properties $vK$ $p$-divisible and $Kv$ perfect. First, we will describe these two properties using the Frobenius endomorphism on ${\rm gr }(K)$. The techniques necessary for this description  will be independent of the theory developed in \cite{GradedFieldsGaloisTheoryBoulagouaz}. Then, we will see how $vK$ $p$-divisible or $Kv$ perfect  interact with the associated module of Kähler differentials of particular valued field extensions.

\section{The graded ring ${\rm gr}(\VR_K)$}\label{GradedRings}

\subsection{The graded ring viewed as a semigroup ring} 
Let $(K, v)$ be a valued field and $\VR_K\subseteq K$ the valuation ring associated to $v$. Consider $v K^{\geq 0}$ the semigroup generated by $v(\VR_K\setminus \{0\})$. Take 
	$${\rm gr}(\VR_K):= \bigoplus_{\gamma\in v K^{\geq 0}}\mathcal{P}_\gamma/ \mathcal{P}_{\gamma}^{+}.$$

In \cite{Matheus}, it is proved that 
${\rm gr}(\VR_K)$ is isomorphic to the semigroup ring $Kv [t^{v K^{\geq 0}}]$, that is, the set of finite formal sums
$$\sum_{i=1}^n b_iv t^{\gamma_i}, \text{ with } a_iv \in  Kv \text{ and } \gamma_i\in v K^{\geq 0} \text{ for each } i=1, \ldots, n $$
where
$$\sum_{i=1}^n b_iv \cdot t^{\gamma_i} + \sum_{i=1}^n b^{'}_iv \cdot t^{\gamma_i} = \sum_{i=1}^{n}(b_i+b^{'}_i)v \cdot t^{\gamma_i}$$
and, in general, the product is not given  by the usual multiplication but by the one defined as follows. A \textbf{choice  function} (or \textbf{section}) on $v K^{\geq 0}$ is any map $\epsilon : v K^{\geq 0} \to \VR_K$ such that $v(\epsilon(\gamma))=\gamma$ for every $\gamma\in v K^{\geq 0}$ (i.e. $\epsilon$ is a right inverse of $v$). We only consider choice functions with $\epsilon(0)=1$. For each choice function, we define the \textbf{twisting}
$$\overline{\epsilon}: v K^{\geq 0} \times v K^{\geq 0} \to  Kv  $$
$$ \overline{\epsilon}(\gamma, \gamma'):=\left(\frac{\epsilon(\gamma)\epsilon(\gamma')}{\epsilon(\gamma+\gamma')}\right)v. $$
The map
$$\times_\epsilon : \vspace{0.2cm}  Kv [t^{v K^{\geq 0}}] \times  Kv[t^{v K^{\geq 0}}] \to  Kv [t^{v K^{\geq 0}}] $$
$$ (a,b)\mapsto a\times_\epsilon b$$
where we set $t^\gamma \times_\epsilon t^{\gamma'} := \overline{\epsilon}(\gamma, \gamma')\cdot t^{\gamma+\gamma'}$ and extend it to  $ Kv [t^{v K^{\geq 0}}]$ in the natural way. This map satisfies the multiplication axioms and turns  $ Kv [t^{v K^{\geq 0}}]$ into a commutative ring. This multiplication is called \textbf{twisted multiplication} induced by $\epsilon$ and we denote the ring by $  Kv [t^{v K^{\geq 0}}]_\epsilon$. 

\begin{Obs}\label{ObsTwistedMult}
	For $a,b\in \VR_K$, by the definition of the twisted multiplication we have
	$$t^{v(ab)} = t^{v(a)+v(b)} = \left(\frac{\epsilon(v(a)+v(b))}{\epsilon(v(a))\epsilon(v(b))}\right)v\cdot t^{v(a)}\times_\epsilon t^{v(b)}. $$
\end{Obs}

\begin{Prop}\cite[Theorem 1.1]{Matheus}
	$${\rm gr}(\VR_K) \cong  Kv [t^{v K^{\geq 0}}]_\epsilon$$ 
	through the isomorphism $$\inv_v(a)\overset{\psi}{\longmapsto} \left(\frac{a}{\epsilon(v(a))}\right)v \cdot t^{v(a)}.$$
\end{Prop}

In the following subsection we give a direct application of the above isomorphism.

\subsection{Frobenius endomorphism on ${\rm gr}(\VR_K) $ } For a ring $R$ with $\Char R=p>0$, the \textbf{Frobenius endomorphism} on $R$ is the map given by $a\mapsto a^p$.  Suppose $\Char Kv = p>0$ and consider the Frobenius endomorphisms $F$ and $\overline{F}$ on ${\rm gr}(\VR_K) $ and $Kv [t^{v K^{\geq 0}}]_\epsilon$  respectively. 
The lemma below is proved using simple computations based on Remark \ref{ObsTwistedMult}.

\begin{Lema}
	The diagram 
	\begin{equation*}
		\begin{tikzcd}[column sep=1cm, row sep = 1cm]
			{\rm gr}(\VR_K) \arrow{r}{F} \arrow[swap]{d}{\psi} & {\rm gr}(\VR_K) \arrow{d}{\psi}\\
			Kv [t^{v K^{\geq 0}}]_\epsilon \arrow[swap]{r}{\overline{F}} &  Kv [t^{v K^{\geq 0}}]_\epsilon
		\end{tikzcd}
	\end{equation*}
is commutative. Hence, $F$ is surjective if and only if $\overline{F}$ is surjective.
\end{Lema}

%\begin{proof}
%	We have
%	\begin{align*}
%		\psi(F(\inv_v(a)))=\psi((\inv_v(a))^p)& =\psi(\inv_v(a^p))\\
%		& = \left(\frac{a^p}{\epsilon(v(a^p))}\right)v \cdot t^{v(a^p)}\\
%		& = \left(\frac{a^p}{\epsilon(v(a^p))} \frac{\epsilon(v(a^p))}{\epsilon(v(a))^p}\right)v \cdot t^{v(a)}\times_\epsilon \cdots \times_\epsilon t^{v(a)}\\
%		& = \left(\frac{a^p}{\epsilon(v(a))^p}\right)v \cdot \left(t^{v(a)}\right)^p
%	\end{align*}
%and
%	\begin{align*}
%	\overline{F}(\psi(\inv_v(a)))=\overline{F}\left(\left(\frac{a}{\epsilon(v(a))}\right)v\cdot t^{v(a)}\right)& = \left(\left(\frac{a}{\epsilon(v(a))}\right)v\cdot t^{v(a)}\right)^p\\
%	& = \left(\left(\frac{a}{\epsilon(v(a))}\right)v\right)^p\left( t^{v(a)}\right)^p\\
%	& = \left(\frac{a^p}{\epsilon(v(a))^p}\right)v \cdot \left(t^{v(a)}\right)^p.
%\end{align*}
%	Therefore the diagram commutes and  hence $F$ is surjective if and only if $\overline{F}$ is surjective.
%\end{proof}

The following result will say that the surjectiveness of the Frobenius endomorphism $F$ on ${\rm gr}(\VR_K) $ can be checked by looking at the valued group and the residue field of $(K,v)$.

\begin{Prop}\label{lemFsurjectivePDivPerf}
Suppose $\Char Kv = p>0$.	Then $F$ is surjective if and only if $v K^{\geq 0}$ is $p$-divisible (hence $v K$ is $p$-divisible) and $Kv$ is perfect. 
\end{Prop}

\begin{proof}  Suppose $F$ is surjective, that is, $\overline{F}$ is surjective. We first prove that  $v K^{\geq 0}$ is $p$-divisible. For $v(a)\in v K^{\geq 0}$, since $\overline{F}$ is surjective we have
	$$t^{v(a)} = \left(  \sum_{i=1}^n b_iv\cdot t^{v(a_i)}  \right)^p  = \sum_{i=1}^n b_i^pv \frac{\epsilon(v(a_i^p)}{\epsilon(v(a_i))^p}\cdot t^{pv(a_i)}  $$ 
	for some $b_1v, \ldots, b_nv \in Kv$ and $a_1, \ldots, a_n\in \VR_K$. Since the elements $t^\gamma$, with $\gamma\in v K^{\geq 0}$, are linearly independent, we must have $t^{v(a)}=t^{pv(a_j)}$ for one and only one index $j$, $1\leq j \leq n$. Hence, $v(a) = pv(a_j)$ and therefore $v K^{\geq 0}$ is $p$-divisible.
	
	\vspace{0.3cm} 
	
	 Now we see that $v K$ is $p$-divisible. If $v(a)<0$, then $-v(a)>0$. Since $v K^{\geq 0}$ is $p$-divisible, there exists $v(a')\geq 0$ such that $-v(a)= pv(a')$, that is, $v(a) = p(-v(a'))$. Thus, $v K$ is also $p$-divisible. 
	 
	 \vspace{0.3cm} 
	 
	 Let us prove that $Kv$ is perfect. The restriction $\overline{F}|_{Kv}$ is an endomorphism on $Kv$ since $(av)^p\in Kv$ for every $av \in Kv$.  We know that 
	 $$av =  \left(  \sum_{i=1}^n b_iv\cdot t^{v(a_i)}  \right)^p  = \sum_{i=1}^n b_i^pv \frac{\epsilon(v(a_i^p))}{\epsilon(v(a_i))^p}v\cdot t^{pv(a_i)}\in Kv  $$ 
	 for some $b_1v, \ldots, b_nv \in Kv$ and $a_1, \ldots, a_n\in \VR_K$, because $\overline{F}$ is surjective. Hence, for one and only one index $j$ we have 
	 $$av = av \cdot t^0= b_j^pv \frac{\epsilon(v(a_j^p))}{\epsilon(v(a_j))^p}v\cdot t^{pv(a_j)}. $$
	 Thus, $pv(a_j)=0$ (which implies $v(a_j)=0$) and $av = b_j^pv \frac{\epsilon(v(a_j^p))}{\epsilon(v(a_j))^p}v$. Since $\epsilon(v(a_j)) = \epsilon(0)=1$, we conclude that $\frac{\epsilon(v(a_j^p))}{\epsilon(v(a_j))^p}v = 1v$ and then
	 $$ av = b_j^pv = (b_jv)^p.$$
	 
	 \vspace{0.3cm} 
	 
	  For the converse, suppose $v K^{\geq 0} $ $p$-divisible  and $Kv$  perfect. It is enough to prove that an element of the form  $bv\cdot t^{v(a)}$ is a $p$-power. Since $v K^{\geq 0}$ is $p$-divisible, there exists $a'\in \VR_K$ such that $v(a)=pv(a')$. Since $Kv$ is   perfect, there exists $b'v \in Kv$ such that
	 $$\frac{\epsilon(v(a'^p))}{\epsilon(v(a'))^p}v bv = (b'v)^p. $$
	Hence,
	\begin{align*}
		(b'v\cdot t^{v(a')})^p & = (b'v)^p (t^{v(a')})^p \\
		& = \frac{\epsilon(v(a'^p))}{\epsilon(v(a'))^p}v \, bv \, \frac{\epsilon(v(a'))^p}{\epsilon(v(a'^p))}v \cdot t^{pv(a')}\\
		& = bv\cdot t^{v(a)}.
	\end{align*} 
Therefore, $\overline{F}$ is surjective, then  $F$ is also surjective.
	\end{proof}

We  will say that a ring $R$ with positive characteristic is \textbf{perfect} if the Frobenius endomorphism on $R$ is surjective. %${\rm gr}(\mathcal{O}_K)$ is \textbf{perfect} if the Frobenius endomorphism $F$ above is surjective.
We note that  ${\rm gr}(\mathcal{O}_K)$ is perfect if and only if  ${\rm gr}(K)$ is perfect. Also, if $K$ is perfect, then ${\rm gr}(K)$ is perfect.

\begin{Obs}\label{obsGradedFields}
Proposition \ref{lemFsurjectivePDivPerf} provides a way of seeing that the surjectiveness of the Frobenius endomorphism on ${\rm gr}(K)$ is equivalent to the definition of perfect graded field given in \cite{GradedFieldsGaloisTheoryBoulagouaz}. 
\end{Obs}

	We can then rephrase Proposition \ref{teoTameiffGradedRingPerfFiniteSeqKP}: $(K,v)$ is a tame field if and only if $(K,v)$ satisfies (FCS) and 	${\rm gr}(K)$ is perfect. This way of seeing the equivalence points out (as in Remark \ref{obsBoulagouazRPerfect}) to the fact that tame extensions can be studied through properties of ${\rm gr}(K)$ and conversely.

\section{Purely inertial and purely ramified extensions}\label{Purely}

In this last section, we give increments for some results of \cite{josneiSpivaKahlerDif}.

\begin{Def}
	Let $(L\mid K, v)$ be a simple algebraic extension of valued fields. We say that:
	\begin{description}
		\item[(i)] $(L\mid K, v)$ is \textbf{purely ramified}  if $vL/vK$ is cyclic and $[L:K]=(vL:vK)$.
		
		\item[(ii)] $(L\mid K, v)$ is \textbf{purely inertial}  if $Lv/Kv$ is simple and $[L:K]=[Lv:Kv]$.
	\end{description}
\end{Def}

If $(L\mid K, v)$ is purely inertial,  take $\eta\in \VR_L$ such that $Lv = Kv(\eta v)$. If $(L\mid K, v)$ is purely ramified, take $\eta\in L$ such that $vL = vK[v(\eta)]$. In both cases, one can show that $L = K(\eta)$. Moreover, the set $\boldsymbol{Q}=\{x\}$ is a complete set for the induced valuation $\nu$ on $K[x]$ \cite[Proposition 5.8]{josneiSpivaKahlerDif}.

\vspace{0.3cm}

Let $\Omega=\Omega_{\VR_L/\VR_K}$ be the \textbf{module of Kähler differentials} for the extension $\VR_L/\VR_K$ (see \cite{kahlerDiffDefinition} for a construction of this module). The study of this module has shown to be crucial to the understanding of deeply ramified fields
% the so called \textit{deeply ramified fields}
 and their relations to the defect (see \cite{cutkoskykuhlmannKahlerDeeplyRamified} and \cite{cutkoskykuhlmannrzepkaIndependetDefect}).

\vspace{0.3cm}

 Suppose $(L\mid K, v)$ purely inertial. We see in \cite[Lemma 6.3 and Proposition 6.4]{josneiSpivaKahlerDif} that $\VR_L = \VR_K[\eta]$ and
$$\Omega \cong \frac{\VR_{ L}}{(g'(\eta))}, $$
where $L=K(\eta)$ and $g$ is the minimal polynomial of $\eta$ over $K$.

\begin{Lema}\label{lemPurelyInertialOmega0Seperable} \cite[Corollary 6.5]{josneiSpivaKahlerDif}
	If $(L\mid K, v)$ is purely inertial, then $\Omega  = (0)$ if and only if $Lv\mid Kv$ is separable.
\end{Lema}

%\begin{proof}From the identification above, $\Omega= (0) $ if and only if $v(g'(\eta))=0$. Take $\eta\in \VR_L$ such that $Lv = Kv(\eta v)$. It follows that $L=K(\eta)$. Let $\overline{h}(y)\in Kv[y] $ be the minimal polynomial of $\eta v$ over $Kv$. Then $v(g'(\eta))=0$ if and only if $\dfrac{d\overline{h}}{dy}\neq 0$, which is equivalent to $\eta v$ being separable over $Kv$. 
%	\end{proof}

The following proposition already  appears in \cite[Theorem 4.3]{cutkoskykuhlmannKahlerDeeplyRamified} and follows directly from Lemma \ref{lemPurelyInertialOmega0Seperable} above.

\begin{Prop}\label{propPurelyInertialOmega0}
	If $Kv$
	%	${\rm gr}(\mathcal{O}_K)$
	is perfect, then every purely inertial extension of $K$ is such that $\Omega=(0)$. 
\end{Prop}

\begin{proof} If
	% ${\rm gr}(\mathcal{O}_K)$ is perfect, then 
	$Kv$ is perfect,
	then all algebraic extensions of $Kv$ are separable. Therefore, if $(L\mid K, v)$ is purely inertial, then $Lv\mid Kv$ is a simple algebraic extension, thus separable. By the Lemma \ref{lemPurelyInertialOmega0Seperable}, we have $\Omega  = (0)$. 	
\end{proof}

Suppose now $(L\mid K, v)$ purely ramified. Take $\eta$ such that $v(\eta)$ generates $vL$ over $vK$. Then $L=K(\eta)$. Consider the set $v(\eta- K)$. 
%is a complete set of key polynomials for $(L\mid K, v)$ ($\nu(f) = v(f(\eta))=\nu_x(f)$ for $\deg(f)<[L:K]$).
 Hence $\gamma=v(\eta)$ is the maximum of $v(\eta - K)$ (since $v(\eta-a)=\nu(x-a)=\nu_x(x-a)\leq \nu(x)=v(\eta)$).
 
 \vspace{0.3cm}
 
  Assume $\gamma>0$. Let $\Delta$ denote the greatest convex 
  subgroup of $vL$ such that $\Delta<\gamma$. Let $g = a_0+a_1x+\ldots +x^n$ be the minimal polynomial of $\eta$ over $K$.

\begin{Lema}\cite[Proposition 6.13]{josneiSpivaKahlerDif} Suppose $(L\mid K, v)$ purely ramified.
	\begin{description}
		\item[(i)] If $\left(\frac{vL}{\Delta}\right)_{>0}$ contains a minimal element, then $\Omega\neq (0)$.
		
		\item[(ii)] If $\left(\frac{vL}{\Delta}\right)_{>0}$ does not contain a minimal element, then $\Omega = (0)$ if and only if there is $l$, $1\leq l \leq n$, such that 
		$$v(l)+v(a_l)-(n-l)\gamma\in \Delta. $$
	\end{description}
\end{Lema}

\vspace{0.2cm}

Consider the following condition.

\vspace{0.2cm}

\begin{description}
	\item[(DRvg)] whenever $\Gamma_1\subsetneq \Gamma_2$ are convex subgroups of $vK$, then $\Gamma_2/\Gamma_1$ is not isomorphic to $\Z$.
\end{description}

\begin{Lema}\label{lemVKpDivDRvg}
	If $vK$ is $p$-divisible, then (DRvg) is satisfied.
\end{Lema}

\begin{proof}Take $\gamma\in \Gamma_2\setminus \Gamma_1$, without lost of generality  $\gamma\geq 0$. Then $\gamma=p\delta$  for some $\delta\in vK^{\geq 0}$. We must have $0\leq \delta < \gamma$. Hence, by the convexity $\delta\in \Gamma_2$. Also, $\delta\not \in \Gamma_1$ since in this case $\gamma+\Gamma_1=p\delta+\Gamma_1=\Gamma_1$ and then we would have $\gamma\in \Gamma_1$, which is not the case. 
	
	By the induced order, $\Gamma_1<\delta+\Gamma_1<\gamma+\Gamma_1$. By repeating this process we construct a strictly decreasing sequence of positive elements in $\Gamma_2/\Gamma_1$. Therefore, $\Gamma_2/\Gamma_1\not\cong \Z$. 
\end{proof}

Using the above lemmas, we can deduce a result similar to Proposition \ref{propPurelyInertialOmega0} for purely ramified extensions. 

\begin{Prop}\label{propPurelyRamifiedOmega0}
	If $vK$ is $p$-divisible (or more generally if (DRvg) is satisfied),
	%${\rm gr}(\mathcal{O}_K)$ is perfect
	 then every purely ramified extension of $K$ is such that $\Omega=(0)$. 
\end{Prop}

\begin{proof} By Lemma \ref{lemVKpDivDRvg}, since $vK$ is $p$-divisible, then (DRvg) is satisfied.   We first see that if we have (DRvg) then $\left(\frac{vL}{\Delta}\right)_{>0}$ does not contain a minimal element. Indeed, suppose $\gamma_m+\Delta>0$ minimal element of  $\left(\frac{vL}{\Delta}\right)_{>0}$. Since $\gamma_m\not\in \Delta$, by convexity it follows that $k\gamma_m\not\in \Delta$ for every $k\in \Z^\ast$. Suppose $\sigma\in vL$ such that $k\gamma_m+\Delta<\sigma+\Delta<(k+1)\gamma_m+\Delta$ for some $k$. Hence, $0+\Delta<\sigma-k\gamma_m+\Delta<\gamma_m+\Delta$, contradicting $\gamma_m$ being the minimal element of $\left(\frac{vL}{\Delta}\right)_{>0}$. This shows us that $\frac{vL}{\Delta}$ is equal to the subgroup generated by $\gamma_m+\Delta$. Then, we define the isomorphism $\phi: \left(\frac{vL}{\Delta}\right)\to \Z$ by $\phi(k\gamma_m+\Delta)=k$, contradicting (DRvg).

	\vspace{0.3cm}
	
	Hence, to see that $\Omega=(0)$ it is sufficient to prove that $v(l)+v(a_l)-(n-l)\gamma\in \Delta $ for some $l$, $1\leq l\leq n$ (by the Lemma above).  
	%Since ${\rm gr}(\mathcal{O}_K)$ is perfect, we have that $vK$ is $p$-divisible. 
	Suppose, aiming for a contradiction, that $p\mid n=[L:K]=(vL:vK)$. Since $vL/vK$ is cyclic, take $\sigma \in vL$ such that $\sigma+vK$ is a generator of the group. We have $n\sigma = \xi\in vK$. Because $vK$ is $p$-divisible, we have $\xi=p\xi'$ for some $\xi'\in vK$. Since $n=pn'$, we have $pn'\sigma=p\xi'$, that is, $ p(n'\sigma-\xi')=0$. This implies $n'\sigma=\xi'\in vK$,  contradicting the order of $vL/vK$ being $n>n'$. Then, $p\nmid n$. Therefore, if we take $l=n$ we conclude that $v(n)+v(a_n)-(n-n)\gamma=0+v(1)+0\gamma=0\in \Delta$. Hence, $\Omega=(0)$.
\end{proof}

Therefore, if $vK$ is $p$-divisible and $Kv$ is perfect (for instance when $(K,v)$ is a tame field), then  every purely ramified extension and every purely inertial extension of $K$ is such that $\Omega=(0)$.

\vspace{0.1cm}

\noindent{\footnotesize CAIO HENRIQUE SILVA DE SOUZA\\
Departamento de Matem\'atica--UFSCar\\
Rodovia Washington Lu\'is, 235\\
13565-905, S\~ao Carlos - SP, Brasil.\\
Email: {\tt caiohss@estudante.ufscar.br}
%
%
%\noindent{\footnotesize JOSNEI NOVACOSKI\\
%Departamento de Matem\'atica--UFSCar\\
%Rodovia Washington Lu\'is, 235\\
%13565-905, S\~ao Carlos - SP, Brasil.\\
%Email: {\tt josnei@ufscar.br} \\\\

%
%\noindent{\footnotesize MARK SPIVAKOVSKY\\
%CNRS UMR 5219 and Institut de Mathématiques de Toulouse \\
%118, rte de Narbonne, 31062 Toulouse cedex 9, France\\
%and\\
%Instituto de Matemáticas (Unidad Cuernavaca) LaSol, UMI CNRS 2001\\
%Universidad Nacional Autónoma de México\\
%Av. Universidad s/n. Col. Lomas de Chamilpa\\
%Código Postal 62210, Cuernavaca, Morelos, México.\\
%Email: {\tt spivakovsky@math.univ-toulouse.fr} \\\\

%\noindent{\footnotesize JOSNEI NOVACOSKI\\
%Departamento de Matem\'atica--UFSCar\\
%Rodovia Washington Lu\'is, 235\\
%13565-905, S\~ao Carlos - SP, Brazil.\\
%Email: {\tt josnei@dm.ufscar.br} \\\\


\begin{thebibliography}{99}
%\bibitem{Naart} M. Alberich-Carrami\~nana, A. Fern\'andez Boix, J. Fern\'andez, J. Gu\`ardia, E. Nart, J. Ro\'e, \textit{Invariants of
%limit key polynomials}, arXiv:2005.04406v1, 2020.
%\bibitem{popescu2} ALEXANDRU, V.; POPESCU, N.; ZAHARESCU, A. A theorem of characterization of residual transcendental extensions of a valuation. \textbf{J. Math. Kyoto Univ.} v. 28, n. 4, pp. 579-592, 1988.
%
%\bibitem{popescu3} ALEXANDRU V.; POPESCU N.; ZAHARESCU, A. Minimal pairs of definition of a residual transcendental extension of a valuation. \textbf{J. Math. Kyoto Univ.} v. 30, n. 2, pp. 207-225, 1990.
%
%\bibitem{popescu4} V. Alexandru, N. Popescu, A. Zaharescu, All valuations on K(x), J. Math. Kyoto Univ. 30 (1990), 281-296.
%%
%\bibitem{Andrei} BENGUS-LASNIER, A., Minimal Pairs, Truncation and Diskoids. arxiv.org/abs/2012.07780.
%
%

\bibitem{AghighKhandujaExSimplyDefectlessNotDefectless}  K. Aghigh, S. K. Khanduja, \textit{On the main invariant of elements algebraic over a henselian valued field},
Proc. Edinb. Math. Soc. 45 (2002), 219–227.



\bibitem{Matheus}  M. S. Barnabé, J. Novacoski, M. Spivakovsky, \textit{On the structure of the graded algebra associated to a valuation}, J. Algebra 560 (2020), 667-679.

%\bibitem{novbarnabeAntigo}  M. dos Santos Barnabé, J. Novacoski, Newton polygons associated to truncated valuations, 2021. 	arXiv:2106.05765 


%\bibitem{novbarnabe}  M. S. Barnabé, J. Novacoski, Valuations on $K[x]$ approaching a fixed irreducible polynomial, J. Algebra 592 (2022), 100-117.




\bibitem{Andrei} A. Bengus-Lasnier, \textit{Minimal Pairs, Truncation and Diskoids}, J. Algebra 579 (2021), 388-427.


\bibitem{AlgebMaxSimplyDefectAnujSudesh} A. Bishnoi, S. K. Khanduja, \textit{On Algebraically Maximal Valued
Fields and Defectless Extensions}, Canad. Math. Bull. 55 (2) (2012), 
233–241.

\bibitem{BoulagouazTeo52} M. Boulagouaz, \textit{The graded and tame extension}. In: Lecture notes in Pure and Applied Mathematics,  n. 153, M. Dekker Inc., New York, 1993.


\bibitem{GradedFieldsGaloisTheoryBoulagouaz} M. Boulagouaz, \textit{An introduction to the Galois theory for graded fields}. In: Lecture notes in Pure and Applied Mathematics, n. 208, M. Dekker Inc., New York, 2000. 



\bibitem{RamificationValuationCutkosky} S. D. Cutkosky,  O. Piltant, \textit{Ramification of valuations}, Advances in Mathematics 183 (2004).


\bibitem{cutkoskykuhlmannKahlerDeeplyRamified} S. D. Cutkosky,  F.-V. Kulhmann, \textit{Kähler differentials of extensions of valuation rings and deeply ramified fields}, arXiv:2306.04967v1 (2023).



\bibitem{cutkoskykuhlmannrzepkaIndependetDefect} S. D. Cutkosky,  F.-V. Kulhmann, A. Rzepka \textit{Characterizations of Galois extensions with independent defect}. arXiv:2306.10022 (2023).

\bibitem{CutkoskyDefectLocalUni} S. D. Cutkosky,  H. Mourtada, \textit{Defect and local uniformization}. Rev. R. Acad. Cienc. Exactas Fís. Nat. Ser. A Mat. 113 (2017), 4211-4226.



\bibitem{spivamahboubkeypoly} J. Decaup, W. Mahboub e M. Spivakovsky, \textit{Abstract key polynomials and comparison theorems with the key polynomials of Mac Lane-Vaquié}, Ill. J. Math. 62 (2018), 253-270.

\bibitem{dutta} A. Dutta, \textit{An invariant of valuation transcendental extensions and its connection with key polynomials}, J.
Algebra 649 (2024), 133–168.

\bibitem{duttakuhlmannTameKeyPoly} A. Dutta, F. -V. Kuhlmann, \textit{Tame key polynomials}, J. Algebra 629 (2022), 162-190.



\bibitem{Eng} A. Engler, A. Prestel,  \textit{Valued Fields}.  Springer-Verlag, Berlin, 2005. 205 p.

\bibitem{almostRingDeeplyRamified} O. Gabber, L. Ramero, \textit{Almost ring theory}, Lecture Notes in Mathematics 1800, Springer-Verlag, Berlin,  2003. 318 p.


\bibitem{spivamahboub}  F. J. Herrera Govantes, W. Mahboub, M. A. Olalla Acosta, M. Spivakovsky, \textit{Key polynomials for simple extensions of valued fields}, J. Singul. 25 (2022), 197-267.

%preprint, arXiv:1406.0657, 2016.
%
%
%

%\bibitem{Spivmahandjul} J. Decaup, M. Spivakovsky and W. Mahboub, \textit{Abstract key polynomials and comparison theorems with the key polynomials of MacLane -- Vaquie}, Illinois J. Math. Vol \textbf{62}, Number \textbf{1-4} (2018), 253 -- 270.

%\bibitem{Kap} I. Kaplansky, \textit{Maximal fields with valuations I}, Duke Math. Journ. \textbf{9}
%(1942), 303 -- 321.



\bibitem{gradedfieldsext}Y. S. Hwang, A. R. Wadsworth, \textit{Algebraic extensions of graded and valued fields}, Communications in Algebra, (2) 27  (1999), 821–840. 
%https://doi.org/10.1080/00927879908826464

\bibitem{KK} H. Knaf, F. -V. Kuhlmann, \textit{Abhyankar places admit local uniformisation in any characteristic}, Ann. Sci. \'Ec. Norm. Sup\'er. (4) 38, 6 (2005), 833-846.

\bibitem{KK2} H. Knaf, F. -V. Kuhlmann, \textit{Every place admits local uniformization in a finite extension of the function field}, Adv. Math. 221, 2 (2009), 428-453.

\bibitem{kulhmannlocalunif} F. -V. Kuhlmann, \textit{Valuation theoretic and model theoretic aspects of local uniformization}. In: Hauser, H., Lipman, J., Oort, F., Quirós, A. (eds) Resolution of Singularities. Progress in Mathematics, vol 181. Birkhäuser, Basel (2000)

\bibitem{kuhlmannResultImmediateNoMax} F.-V. Kuhlmann, \textit{A classification of Artin-Schreier defect extensions and characterizations of defectless fields}, Ill. J. Math. 54, 2 (2010), 397-448.


\bibitem{kuhlmannTameFields} F.-V. Kuhlmann, \textit{The algebra and model theory of tame valued fields}, J. Reine Angew. Math.  719 (2016), 1-43. %https://doi.org/10.1515/crelle-2014-0029

\bibitem{DeeplyRamifiedKR} F-V. Kuhlmann, A. Rzepka, \textit{The valuation theory of deeply ramified fields and its connection with defect extensions}, Trans. Amer. Math. Soc. 376 (2023), 2693-2738
%%



%\bibitem{leloup} G. Leloup, Key polynomials, separate and immediate valuations, and simple extensions of valued fields.   	arXiv:1809.07092, 2019.

\bibitem{MacLane} S.  MacLane, \textit{A construction for absolute values in polynomial rings}, Trans. Amer. Math. Soc. 40, 3 (1936), 363-395.

\bibitem{kahlerDiffDefinition} H. Matsumura, \textit{Commutative Algebra}, Benjamin/Cummings Publishing Co., Reading, Mass
(1970).


\bibitem{NartKeyPolyValuedFields} E. Nart, \textit{Key polynomials over valued fields}. Publ. Mat. 64 (2020), 3-42.
%
\bibitem{Nart} E. Nart, \textit{MacLane-Vaqui\'e chains of valuations on a polynomial ring}, Pacific J. Math. 311, 1 (2021), 165-195.
%arXiv:1911.01714v2, 2019.


\bibitem{NartJosneiDefectFormula} E. Nart, J. Novacoski, \textit{The defect formula}, Adv. Math. 482 (2023), 1-44. 

\bibitem{josneiKeyPolyMinimalPairs}  J. Novacoski, \textit{Key polynomials and minimal pairs}, J. Algebra 523 (2019), 1-14.

\bibitem{josneimonomial} J. Novacoski, \textit{On MacLane-Vaqui\'e key polynomials}, J. Pure Appl. Algebra 225 (2021), 106644.



%
%\bibitem{josneicaio} J. Novacoski, C. H. Silva de Souza, On truncations of valuations, J. Pure Appl. Algebra 226 (2022) 106965. 
%Disponível em: https://doi.org/10.1016/j.jpaa.2021.106965. Acesso em: 26 nov. 2021.




\bibitem{JN_1} J. Novacoski, M. Spivakovsky, \textit{Reduction of local uniformization to the rank one case}, Valuation Theory in Interaction, EMS Series of Congress Reports (2014), 404-431.


\bibitem{josneiKeyPolyPropriedades}   J. Novacoski, M. Spivakovsky, \textit{Key polynomials and pseudo-convergent sequences}, J. Algebra 495 (2018), 199-219.
%

\bibitem{josneiSpivaKahlerDif}   J. Novacoski, M. Spivakovsky, \textit{Kähler differentials, pure extensions and minimal key polynomials},  (2023), arXiv:2311.14322.

%






%
%\bibitem{Saturnino} J-C. San Saturnino, Defect of an extension, key polynomials and local uniformization, J. Algebra 481 (2017), 91-119.

%
%
\bibitem{Tei} B. Teissier, \textit{Valuations, deformations, and toric geometry}. Valuation theory and its applications, Vol. II (Saskatoon, SK, 1999), Fields Inst. Commun. \textbf{33} (2003), 361--459.

\bibitem{Vaq}   M. Vaquié, \textit{Extension d'Une Valuation}, Trans. Amer. Math. Soc.  359, 7 (2007), 3439-3481.

 
 
  \bibitem{Vaqfamilleadmissible} M. Vaqui\'e, \textit{Famille admissible de valuations et defaut d'une extension}, J. Algebra 359 no. 2 (2007), 859-876.
%  
    \bibitem{Vaq3} M. Vaqui\'e, \textit{Famille admise associée à une valuation de K[x]}, in: Singularités franco-japonaises, in: Séminaires
    et Congrès,  10, Soc. Math. France, 2005.
%
%\bibitem{popescuSurLaDefinition} POPESCU, L.; POPESCU, N. Sur la definition des prolongements résiduels transcendents d'une valuation sur un corps K à K(x). \textbf{Bulletin mathématique de la Société des Sciences Mathématiques de la République Socialiste de Roumanie},
%Nouvelle Série, v. 33 (81), n. 3, pp. 257-264, 1989.


%\bibitem{zar} O. Zariski; P. Samuel.  Commutative Algebra II. Springer-Verlag Berlin Heidelberg, 1960. 416 p.

\end{thebibliography}
\end{document}